\documentclass[oneside,10pt,leqno]{amsart}
\usepackage{amssymb,amsmath,latexsym,amscd}
\usepackage{array,hhline}

\usepackage{longtable}

\newcommand{\diag}{\operatorname{diag}}

\def\gg{\mathfrak{g}}
\def\gh{\mathfrak{h}}

\def\gj{\mathfrak{j}}
\def\gk{\mathfrak{k}}

\def\gm{\mathfrak{m}}
\def\gn{\mathfrak{n}}

\def\gr{\mathfrak{r}}
\def\gs{\mathfrak{s}}
\def\gt{\mathfrak{t}}

\def\gv{\mathfrak{v}}
\def\gw{\mathfrak{w}}

\def\gz{\mathfrak{z}}

\def\A{\mathbb{A}}

\def\C{\mathbb{C}}
\def\D{\mathbb{D}}
\def\E{\mathbb{E}}
\def\F{\mathbb{F}}

\def\H{\mathbb{H}}
\def\I{\mathbb{I}}

\def\O{\mathbb{O}}

\def\Q{\mathbb{Q}}
\def\R{\mathbb{R}}
\def\bbS{\mathbb{S}}
\def\T{\mathbb{T}}

\def\Z{\mathbb{Z}}

\def\bI{\mathbf{I}}

\def\Im{{\rm Im}\,}
\def\Re{{\rm Re}\,}

\def\Ad{{\rm Ad}}

\def\rank{{\rm rank}\,}

\def\Ker{{\rm Ker\,}}
\def\tr{{\rm trace\,}}
\def\Aut{{\rm Aut}}
\def\Int{{\rm Int}}
\def\Out{{\rm Out}}

\def\diag{{\rm diag}}

\renewcommand{\thesection}{\arabic{section}}
\renewcommand{\thesubsection}{\thesection\Alph{subsection}}

\renewcommand{\thetable}{{\large \thesection.\arabic{equation}}}
\newtheorem{theorem}[equation]{Theorem}

\newtheorem{lemma}[equation]{Lemma}

\newtheorem{corollary}[equation]{Corollary}

\newtheorem{proposition}[equation]{Proposition}

\newtheorem{examples}[equation]{Examples}

\newtheorem{remark}[equation]{Remark}

\def\sideremark#1{\ifvmode\leavevmode\fi\vadjust{\vbox to0pt{\vss
 \hbox to 0pt{\hskip\hsize\hskip1em
\vbox{\hsize2cm\tiny\raggedright\pretolerance10000 
 \noindent #1\hfill}\hss}\vbox to8pt{\vfil}\vss}}} 
\newcommand \<{\langle}
\renewcommand \>{\rangle}

\begin{document}

\title{On the Homogeneity Conjecture}

\author{Joseph A. Wolf}\thanks{Research partially supported by a Simons
Foundation grant}
\address{Department of Mathematics \\ University of California, Berkeley \\
	CA 94720--3840, U.S.A.} \email{jawolf@math.berkeley.edu}

\date{file last edited 28 March 2023}

\dedicatory{Dedicated to Prof. Toshi Kobayashi on the occassion of his 60th birthday}



\begin{abstract}
Consider a connected homogeneous Riemannian manifold $(M,ds^2)$ and a
Riemannian covering $(M,ds^2) \to \Gamma \backslash (M,ds^2)$.  If
$\Gamma \backslash (M,ds^2)$ is homogeneous then every $\gamma \in \Gamma$
is an isometry of constant displacement.  The Homogeneity Conjecture 
suggests the converse: if every $\gamma \in \Gamma$ is an isometry of 
constant displacement on $(M,ds^2)$ then $\Gamma \backslash (M,ds^2)$ is 
homogeneous.  We survey the cases in which the Homogeneity Conjecture
has been verified, including some new results, and suggest some related 
open problems.
\end{abstract}

\maketitle

\section{\bf Introduction.}\label{sec1}
\setcounter{equation}{0}

Let $(M',ds'^2)$ be a connected locally homogeneous Riemannian manifold.  
We are going to study a simple geometric condition for $(M',ds'^2)$ to be 
(globally) homogeneous.  The obvious conditions for this are (i) that 
$(M',ds'^2)$ is complete and (ii) that the universal Riemannian covering
manifold $(M,ds^2)$ of $(M',ds'^2)$ is complete, connected and locally 
homogeneous.  Then $(M,ds^2)$ is homogeneous, specifically

\begin{lemma}\label{cover-homog}
If $\xi_U$ is a Killing vector field on an open subset $(U,ds^2|_U)$ of
$(M,ds^2)$, then $\xi_U$ extends uniquely to a Killing vector field $\xi$
on $(M,ds^2)$, and $\xi$ generates a one parameter group 
$\{\exp(t\xi) | t \in \R\}$ of
isometries of $(M,ds^2)$.  In particular local homogeneity of $(M',ds'^2)$
results in global homogeneity of $(M,ds^2)$.
\end{lemma}

The special case $U = M$ of Lemma \ref{cover-homog} says

\begin{lemma}\label{lift-isom}
Suppose that $(M',ds'^2)$ is homogeneous.  Consider the largest connected 
group of isometries, $G' = \bI^0(M',ds'^2)$.  Then there is a Lie group 
covering $G \to G'$ such that the action of $G'$ on $M'$ lifts to an
effective transitive isometric action of $G$ on $(M,ds^2)$.
\end{lemma}

Now we can formulate a simple, but basic, observation.

\begin{proposition}\label{CW-translation}
{\rm (\cite[Theorems 1 and 2]{W1959}.)}
Suppose that $(M',ds'^2)$ is homogeneous.  Express $M' = \Gamma\backslash M$
where $\Gamma$ is a discontinuous group of fixed point free isometries of
$(M,ds^2)$.  Let $G \subset \bI(M,ds^2)$ as in Lemma {\rm \ref{lift-isom}}.
Then $G$ centralizes $\Gamma$, and every $\gamma \in \Gamma$ is an isometry 
of constant displacement on $(M,ds^2)$.
\end{proposition}

\begin{proof} As constructed, every element of $G$ sends $\Gamma$--orbits 
to $\Gamma$-orbits, so $G$ normalizes $\Gamma$.  But $G$ is connected and 
$\Gamma$ is discrete, so $G$ centralizes $\Gamma$.

Let $x, y \in M$ and $\gamma \in \Gamma$.  Choose $g \in G$ with 
$g(x) = y$.  Let $\rho$ denote distance in $(M,ds^2)$.  Then the 
displacement $\rho(x,\gamma(x)) =
\rho(gx,g\gamma(x)) = \rho(gx,\gamma g(x)) = \rho(y,\gamma(y))$.
\end{proof}

The ``Homogeneity Conjecture'' includes the converse:

\begin{quote}
\noindent {\bf Homogeneity Conjecture.}
{\sl Let $(M,ds^2)$ be a connected simply connected homogeneous Riemannian 
manifold and $\pi: (M,ds^2) \to (M',ds'^2)$ a Riemannian covering.  Express
$M' = \Gamma \backslash M$ where $\Gamma$ is a discontinuous group of isometries
of $(M,ds^2)$.  Then $(M',ds'^2)$ is homogeneous if and only if every
$\gamma \in \Gamma$ is an isometry of constant displacement on
$(M,ds^2)$.}
\end{quote}

This paper surveys the published cases for which the Homogeneity Conjecture has 
been verified, and also includes a number of new results.  In the cases where 
I recall published results I tried to indicate the steps in the proof, with 
precise references to the original papers for the reader who wants to see
the details.  In the cases of new results I indicated complete proofs.  
I also tried to be precise about cases where verification of the Homogeneity 
Conjecture still is an open problem.

In order to have uniform notation to the extent possible, some notation has 
been changed from that in the references.  This is especially noticeable 
when dealing with isotropy split fibrations.  However the structure of the 
isometry group and the definition of the group $G$
can change from section to section as we look at various classes of 
Riemannian manifolds.

In {\bf Part I} we sketch the verification of the Homogeneity Conjecture for 
Riemannian symmetric spaces, and then we indicate how this verification
extends to Finsler symmetric spaces.  Part I consists of 

{\bf Section \ref{sec2}: Spaces of Constant Curvature.}  These were the first 
instances of results that led to the Homogeneity Conjecture. The euclidean 
case (curvature $K = 0$) 
is elementary. The hyperbolic case ($K < 0$) illustrates the fact that 
for noncompact manifolds one can work more generally with isometries of 
bounded displacement; it depends on the fact there that any two distinct 
geodesics diverge. The elliptic case ($K > 0$) is the hard case 
because the round sphere has so many symmetries, and it requires
some nontrivial finite group theory.

{\bf Section \ref{sec3}: Riemannian Symmetric Spaces with Simple Isometry 
Group.}  As one 
expects from the cases of constant curvature, 
the noncompact case is straightforward and the compact case is
delicate.  It is relatively straightforward to study isometries $\gamma$ 
of constant displacement when $\gamma$ is contained in the identity  
component $G = \bI^0(M,ds^2)$ of the isometry group, but is much less 
straightforward when $\gamma$ is in a component that involves outer 
automorphisms of $G$.

{\bf Section \ref{sec4}: Riemannian Symmetric Spaces that are Group Manifolds.} 
These are 
the ones where $\bI^0(M,ds^2)$ is not simple and the symmetric space 
$(M,ds^2)$ is a group manifold $G$ with bi-invariant Riemannian 
metric $ds^2$.  Then $M = (G \times G)/\{diag\, G\}$; $G$ acts by left and 
right translations and the isotropy subgroup of $\bI^0(M,ds^2)$ acts
on $G$ by inner automorphisms. The delicate points are the cases where $G$ has
outer automorphisms.

{\bf Section \ref{sec5}: A Classification Free Approach.} In this later 
development, 
isometries of constant displacement are characterized as preserving
a minimizing geodesic from a point to its image.  Combining this with
Proposition \ref{dS} one proves that if an isometry $g \in \bI(M,ds^2)$
is of constant displacement then its centralizer in $\bI(M,ds^2)$
is transitive on $M$.  That is enough to verify the Homogeneity
Conjecture for $\Gamma$ cyclic, and it sidesteps the case by case argument
parts of the proofs of Theorems \ref{symm-simple-isom} and \ref{symm-group}
that involve close looks at outer automorphisms.

{\bf Section \ref{sec6}: Extension to Finsler Symmetric Spaces.} 
It makes perfect sense to 
consider isometries of constant displacement on a metric space $(M,\rho)$,
and if $\Gamma \backslash (M,\rho)$ is a homogeneous metric space then every
$\gamma \in \Gamma$ is an isometry of constant displacement.  Thus one can
consider the Homogeneity Conjecture for metric spaces, in particular for
Finsler manifolds.  In this section we indicate the proof of the Homogeneity
Conjecture for Finsler symmetric spaces.  

In {\bf Part II} we sketch the verification (or progress toward the 
verification) 
of the Homogeneity Conjecture for various geometrically defined 
classes of compact homogeneous Riemannian manifolds.  Part II consists of

{\bf Section \ref{sec7}: Isotropy Splitting Fibrations.}  This section
develops a tool for tracing constant displacement isometries along a certain
class of fibrations, modeled on the canonical projections
$SO(k+\ell)/SO(k) \to SO(k+\ell)/[SO(k) \times SO(\ell)]$ of a Stieffel
manifold over a Grassmann manifold.  

{\bf Section \ref{sec8}: Manifolds of Positive Euler Characteristic.}  
This section consists of new results.
We look at Riemannian manifolds $(M,ds^2)$ where $M = G/H$, $G$ is a 
compact connected Lie group, $ds^2$ is a $G$--invariant Riemannian 
metric on $M$, and the Euler characteristic $\chi(M) \ne 0$.
This last condition is equivalent to $\rank H = \rank G$, and then
$\chi(M)$ is the quotient $|W_G/W_H|$ of the orders of the Weyl groups.
The proofs quickly reduce to the case where $G$ is simple.
The main result here is Theorem \ref{euler-pos}.  That verifies the 
Homogeneity Conjecture when, for every $\gamma \in \Gamma$, $\Ad(\gamma)$
is an inner automorphism on $G$. It is an open problem to deal with outer
automorphisms.

{\bf Section \ref{sec9}: Compact Group Manifolds.}  This section also 
contains new results.  We look at Riemannian manifolds $(M,ds^2)$ on
which a connected Lie group $G$ acts simply transitively by isometries.
That reduces to Riemannian manifolds $(G,ds^2)$ where $ds^2$ is a
Riemannian metric invariant under the left translations $\ell(g), g \in G$.
The identity component of the isometry group has form $\ell(G)\times r(H)$
where $r(H)$ consists of the right translations $r(h)$ that preserve $ds^2$, 
i.e. such that $\Ad(h)$ preserves the inner product on $\gg$ defined by $ds^2$.
The main result here is Theorem \ref{result-grp}, which says that a finite
group $\Gamma$ of constant displacement isometries is contained either in
$r(H)$ or $\ell(G)$.  In the $r(H)$ case $\Gamma$ centralizes $\ell(G)$, so
$\Gamma \backslash (G,ds^2)$ is homogeneous and the Homogeneity Conjecture
is verified.  In the $\ell(G)$ case one needs more information on the
elements of $\Gamma$.

{\bf Section \ref{sec10}: Positive Curvature Manifolds.}  In this section 
we verify the Homogeneity Conjecture for Riemannian homogeneous spaces
$(M,ds^2)$, $M = G/H$, such that $M$ admits Riemannian metric $dt^2$ of
strictly positive sectional curvature.  We rely on the classification of
those spaces $M$ and the structure of their isometry groups, carry that
over to the more general spaces $(M,ds^2)$, and use various tools to 
complete the verification there of the Homogeneity Conjecture.

In {\bf Part III} we sketch the verification of the Homogeneity Conjecture 
for several classes of noncompact homogeneous Riemannian manifolds.  
In these noncompact cases ``bounded'' can replace ``constant displacement''
and the result becomes independent of the choice of Riemannian metric.
The results are all based on \cite{W1964} and \cite{T1964}. 
Part III consists of

{\bf Section \ref{sec11}: Negative Curvature.} This section recalls the results
of \cite{W1964}, which were applied to Riemannian symmetric spaces in
Subsection \ref{ss3a}A.

{\bf Section \ref{sec12}: Semisimple Groups.}  This section deals with the 
cases where a real semisimple group $G$, without any compact factors, is
transitive on $(M,ds^2)$.  Combining ideas from \cite{W1964} and \cite{T1964}
it is shown that $(M,ds^2)$ has no nontrivial bounded isometries, and in 
particular the Homogeneity Conjecture is verified for $(M,ds^2)$.  This
applies, in particular, to the flag domains that appear in automorphic
function theory.

{\bf Section \ref{sec13}: Bounded Automorphisms.}  This section recalls 
the results of Jacques Tits \cite{T1964} on bounded automorphisms.  They
give results on bounded isometries.  For reasons of clarity I have quoted 
them in the original.

{\bf Section \ref{sec14}: Exponential Solvable Groups.}  This section
combines results on exponential solvable groups (\cite{W2017}, \cite{W2022} 
and semisimple groups with no compact factors (from Section \ref{sec12}.  
It is shown that if such a group is transitive then there are no nontrivial 
bounded isometries.

In {\bf Part IV} we summarize the results and pose some open problems.
Part IV consists of 

{\bf Section \ref{sec15}: Open Problems.}  This section mentions five open
problems related to the Homogeneity Conjecture.  They are
\begin{itemize}
\item to complete the results on manifolds $(M,ds^2)$ with $\chi(M) > 0$,
\item to complete the results on group manifolds with left invariant 
Riemannian metric,
\item to verify the Homogeneity Conjecture for weakly symmetric Riemannian 
manifolds, or even geodesic orbit Riemannian manifolds,
\item to study the Homogeneity Conjecture for Finsler manifolds, and
\item to study an appropriate variation on the Homogeneity Conjecture for 
pseudo--Riemannian manifolds.
\end{itemize}

Much of the material in Sections \ref{sec8}, \ref{sec9} and \ref{sec14} is
new, except of course where it is cited from one of the references.

\medskip
\centerline{\bf Part I. Riemannian Symmetric Spaces.}
\medskip

In three Sections \ref{sec2} through \ref{sec4} we will sketch the proof 
of the Homogeneity Conjecture for the cases where $(M,ds^2)$ is a 
Riemannian symmetric space.  In Section \ref{sec2} we carry this out 
for the cases of constant sectional curvature; they illustrate the issues 
that must be addressed in general, and in particular
for symmetric spaces.  In fact the case of constant positive curvature
is the most difficult case, and it requires a bit of finite group theory.

After I published these results on Riemannian symmetric spaces,
H. Freudenthal and V. Ozols gave shorter proofs for some special cases.
Freudenthal \cite{F1963} gave a proof for the case where $\Gamma$ is
contained in the identity component $\bI^0(M,ds^2)$, and Ozols \cite{O1974}
gave a classification--free proof for the case where $\Gamma$ is cyclic.
The result of Ozols is sketched in Section \ref{sec5}.

Section \ref{sec6} sketches the proof of the Homogeneity
Conjecture for Finsler symmetric spaces.  This is a bit technical; the
strategy is to develop tools of Finsler geometry that allow one to reduce 
considerations to the Riemannian case.

The main results in Part I are Theorem \ref{hc-const-curv} for constant
curvature spaces, Theorem \ref{hc-symmetric} (for Riemannian symmetric spaces)
and Theorem \ref{hc-finsler} (for Finsler symmetric spaces).

\section{\bf Spaces of Constant Curvature.}\label{sec2}
\setcounter{equation}{0}
\setcounter{subsection}{1}

In this section we indicate the proof of the Homogeneity Conjecture for 
the cases where $(M,ds^2)$ has constant sectional curvature $K$.  This
material comes from \cite{W1959} and \cite{W1960}.  The statement is

\begin{theorem}\label{hc-const-curv}
Let $(M,ds^2)$ be a connected simply connected homogeneous Riemannian 
manifold of constant sectional curvature $K$.  Let 
$\pi: (M,ds^2) \to (M',ds'^2) = (\Gamma \backslash M, ds'^2)$ be a
Riemannian covering.  Then $(M',ds'^2)$ is homogeneous if and only if every
$\gamma \in \Gamma$ is an isometry of constant displacement on $(M,ds^2)$.
\end{theorem}

Theorem \ref{hc-const-curv} will follow directly from Proposition
\ref{CW-translation}, Lemmas \ref{const-curve-neg} and \ref{const-curve-zero},
and Proposition \ref{const-curve-pos}.

\smallskip
\centerline{\bf \thesubsection.  Constant Nonpositive Curvature.}

\noindent The hyperbolic space 
case $K < 0$ is easy because any two distinct geodesics diverge.  If
$\gamma$ is an isometry of bounded displacement and $g$ is a geodesic of
$(M,ds^2)$, this says that $\gamma(g) = g$. Since any point $x \in M$ can
be described as the intersection of two geodesics this says $\gamma(x) = x$.
Thus

\begin{lemma}\label{const-curve-neg}
Any isometry of bounded displacement on real hyperbolic space is the
identity transformation.  If $(M',ds'^2)$ is a homogeneous Riemannian manifold
of constant negative curvature, then it is isometric to the real hyperbolic 
space $\H^n(\R)$.
\end{lemma}

The euclidean space case $K = 0$ is elementary because any two straight lines
diverge unless they are parallel.  If
$\gamma$ is an isometry of bounded displacement and $\sigma$ is a geodesic of
$(M,ds^2)$, this says that $\gamma(\sigma)$ is parallel to $\sigma$.  A rigid 
motion of euclidean space that sends every straight line to a parallel line 
must be a pure translation.  Thus

\begin{lemma}\label{const-curve-zero}
Any isometry of bounded displacement on euclidean space is a pure translation.
If $(M',ds'^2)$ is a homogeneous Riemannian manifold
of constant sectional curvature zero, then it is isometric to a product 
$\Gamma \backslash \E^n \cong \T^k \times \E^{n-k}$ of a locally 
euclidean torus with an euclidean space.
\end{lemma}

\smallskip
\addtocounter{subsection}{1}
\centerline{\bf \thesubsection.  Constant Positive Curvature: Binary Dihedral 
and Binary Polyhedral Groups.}

\noindent 

The round sphere case $K > 0$ is the hard case.  It is worked out in 
\cite{W1960} and uses some nontrivial finite group theory \cite{S1955}.
It was more or less conjectured by G. Vincent, at least for the 
case where $\Gamma$ is cyclic or binary dihedral, in the last
sentence of \cite[\S 10,5]{V1947}.  Here is a sketch of the proof
from \cite{W1960}.

We first describe the homogeneous quotients of $\bbS^n$ so that we know
the structure of the groups $\Gamma$ that occur here.  We'll need some 
finite group preliminaries.  

Recall that the {\em dihedral group} $\D_m$
has order $2m > 4$ and is the symmetry group in $SO(3)$ of a regular $m$--gon,
the {\em tetrahedral group} $\T$ has order $12$ and is the symmetry 
group in $SO(3)$ of a regular tetrahedron, the {\em octahedral group} $\O$
has order $24$ and is the symmetry group in $SO(3)$ of a regular 
octahedron, and the {\em icosahedral group} $\I$ has order $60$ and 
is the symmetry group in $SO(3)$ of a regular icosahedron.
The last three are the {\em polyhedral groups}.
Every finite subgroup of $SO(3)$ is a cyclic, dihedral, tetrahedral, 
octahedral or icosahedral group.  If two finite subgroups of $SO(3)$ are 
isomorphic they are conjugate in $SO(3)$.

Let $p: Sp(1) \to SO(3)$ denote the universal covering group.  It is 
$2$--sheeted, and $Sp(1)$ is the group of unit quaternions.  The 
{\em binary dihedral}, {\em binary tetrahedral}, {\em binary octahedral}
and {\em binary icosahedral} groups are the
\begin{equation*}
\D^*_k = p^{-1}(\D_k),\;\; \T^* = p^{-1}(\T),\;\; \O^* = p^{-1}(\O) 
\text{ and } \I^* = p^{-1}(\I).
\end{equation*}
The last three are the {\em binary polyhedral} groups.
Every finite subgroup of $Sp(1)$ is a cyclic, binary dihedral, binary 
tetrahedral, binary octahedral or binary icosahedral group.  
If two finite subgroups of $Sp(1)$ are
isomorphic they are conjugate in $Sp(1)$.

Suppose that $(M',ds'^2) = \Gamma \backslash \bbS^n$ is a homogeneous 
Riemannian manifold of constant sectional curvature $K > 0$.  As in
Lemma \ref{lift-isom} the identity component $G' = \bI^0(M',ds'^2)$
of the centralizer of $\Gamma$ in the isometry group $\bI(\bbS^n) = O(n+1)$
is transitive on $\bbS^n$.  In particular it is irreducible on the ambient 
$\R^{n+1}$.  Thus Schur's Lemma says that the centralizer $\A$ of $G'$
in the algebra of linear transformations of $\R^{n+1}$ is a real division
algebra.  Now there $\R^{n+1}$ is a left $\A$ vector space,
$\Gamma \subset \A \cap O(n+1)$, and there are three cases.
\begin{equation}\label{div-alg}
\begin{aligned}
 {\bf 1.} \text{ $\A = \R$,} &\text{ so $\Gamma \subset (\R \cap O(n+1)) = 
		O(1) = \{z \in \R \mid |z = 1\} = \{\pm I\}$.}\\  
	&\text{ Then $(M',ds'^2)$ is the round sphere or the real
		projective space.}\\
 {\bf 2.} \text{ $\A = \C$,} &\text{ so $\Gamma \subset (\C \cap O(n+1)) = 
	U(1) = \{z \in \C \mid |z| = 1\}$.} \\
	&\text{ Then $\Gamma$ is a
		cyclic group $\Z_k,  k > 2$,
		with generator $\diag\{J_k, \dots , J_k\}$}\\
	&\text{	\quad where $J_k = \left ( \begin{smallmatrix}
		e^{2\pi\sqrt{-1}/k} & 0 \\ 0 & e^{-2\pi\sqrt{-1}/k}
		\end{smallmatrix} \right )$,
		and $(M',ds'^2)$ is a ``lens space''.} \\
 {\bf 3.} \text{ $\A = \Q$,} &\text{ so $\Gamma \subset (\Q \cap O(n+1)) = 
	Sp(1) = \{z \in \Q \mid |z| = 1\}$. }\\ 
	&\text{ Then $\Gamma \subset Sp(1)$ is a binary
		dihedral or binary polyhedral group. } 
\end{aligned}
\end{equation}

\smallskip
\addtocounter{subsection}{1}
\centerline{\bf \thesubsection.  Constant Positive Curvature.}

\noindent
Now we assume that $\Gamma$ is a finite non--cyclic group that has a 
faithful unitary representation $\varphi: \Gamma \to U(\ell)$ such that 
every $\varphi(\gamma)$ is an isometry of constant displacement on the 
unit sphere $\bbS^{2\ell - 1}$ in $\C^\ell$.

\begin{lemma}\label{clif-basics}
{\rm (\cite[Lemma 1]{W1960})}
With $\Gamma$ as just specified,

{\rm (1)} Every abelian subgroup of $\Gamma$ is cyclic.

{\rm (2)} Given primes p and q, every subgroup of $\Gamma$ of order pq is cyclic.

{\rm (3)} $\Gamma$ has a unique element of order 2. It generates the center of 
$\Gamma$.

{\rm (4)} If $\alpha$ and its transpose $\alpha^t$ are conjugate elements of 
$\Gamma$, then either $\alpha =  \alpha^t$ or $\alpha^{-1} =  \alpha^t$.
\end{lemma}

\noindent {\em Sketch of Proof.}  Statements (1), (2) and the uniqueness of 
elements of order $2$ in $\Gamma$ follow from the fact that $\Gamma$ has a 
free action on a sphere. As $\Gamma$ has even order \cite[\S 10.5]{V1947}, 
(3) follows when we show that any central element $\gamma \in \Gamma$ has 
order 2.

Looking at characters, one sees that the irreducible components of 
$\varphi$ are equal and inherit the property that every element in the
image is an isometry of constant displacement on the unit sphere in the
representation space.  Thus we may assume that $\varphi$ is irreducible. 
If $\gamma \ne 1$ is central in $\Gamma$, Schur's Lemma shows that 
$\varphi(\gamma)$ is scalar, so its eigenvalues satisfy 
$\lambda = \overline{\lambda}$, in other words $\lambda = \pm 1$.
As $\gamma \ne 1$ now $\varphi(\gamma) = -1$.  That proves (3).

If $\alpha$ and $\alpha^t$ are conjugate they have the same eigenvalues,
$\{\lambda, \overline{\lambda}\} = \{\lambda', \overline{\lambda'}\}$.
If $\lambda = \lambda'$ then $\alpha = \alpha^t$, and if $\lambda =
\overline{\lambda'}$ then $\alpha^{-1} = \alpha^t$.  That proves (4).
\hfill $\diamondsuit$

\begin{lemma}\label{build} 
{\rm (\cite[Lemma 2]{W1960})}
Let $\Gamma_1$ be a normal subgroup of $\Gamma$, assume
$\Gamma_1$ cyclic or binary dihedral $\D_k^*$ ($k \ne 2$), and suppose 
$\Gamma$ generated by $\Gamma_1$ and some element $\tau \in \Gamma$. Then 
$\Gamma$ is cyclic or binary dihedral
\end{lemma}

\begin{proof}  This is a typical verification based on the
structure of binary dihedral and binary polyhedral groups in terms of generators
and relations.  First, one supposes that $\Gamma_1 = \<\alpha\>$, cyclic.
Then $\tau\alpha\tau^{-1}$ is $\alpha$ or $\alpha^{-1}$ by Lemma 
\ref{clif-basics}.  If $\tau\alpha\tau^{-1} = \alpha$ then $\Gamma$ is abelian,
hence cyclic by Lemma \ref{clif-basics}.  If $\tau\alpha\tau^{-1} = 
\alpha^{-1} \ne \alpha$ then $\tau$ has order $4$ and $\Gamma$ is binary
dihedral.

Now suppose $\Gamma_1 = \D^*_m$ with $m > 2$: $\alpha^m = \beta^4 = 1$ and
$\beta\alpha\beta^{-1} = \alpha^{-1}$.  As $m \ne 2$, $\<\alpha\>$ is a 
characteristic subgroup of $\Gamma_1$, hence normal in $\Gamma$, so 
$\tau\alpha\tau^{-1}$ is $\alpha$ or $\alpha^{-1}$.  $\beta^2$ is central in
$\Gamma$ because it has order $2$.  Thus 
$\Gamma' := \<\alpha,\beta^2,\tau\>$ or
$\Gamma' := \<\alpha,\beta^2,\tau\beta\>$ is abelian, hence cyclic, and
$\Gamma = \<\Gamma',\beta\>$.  $\tau\beta\tau^{-1}$ has order $4$ so it has
form $\beta\alpha^u$ or $\beta^3\alpha^u$.  Thus $\beta^{-1}\tau\beta$ has
form $\alpha^u\tau$ or $\alpha^u\tau\beta^2$, and $\beta^{-1}(\tau\beta)\beta$
has form $\alpha^u(\tau\beta)$ or $\alpha^u(\tau\beta)\beta^2$.  Thus $\Gamma_1$
is normal in $\Gamma$, and the first paragraph of the proof shows that
$\Gamma$ is binary dihedral.
\end{proof}

Now we need a result of G. Vincent \cite[Théorème X]{V1947} which implies 
that if $\Gamma$ has all Sylow subgroups cyclic then it is either cyclic or 
binary dihedral $\D_m^*$ (m odd).  We will also need a procedure of 
H. Zassenhaus \cite[proof of Satz 7]{Z1935}, which depends on his result 
\cite[Satz 6]{Z1935}: If $\Gamma$ is solvable and of order not divisible
by $2^{s+1}$, and if $\Gamma$ has an element of order $2^{s-1}, s > 1$, 
then $\Gamma$ has a normal subgroup $\Gamma_1$ with cyclic Sylow $2$--subgroup,
such that $\Gamma/\Gamma_1$ is the cyclic group $\Z_2$, the alternating group
$\A_4$, or the symmetric group $\bbS_3$\,.  While \cite{Z1935} has errors,
they are corrected in \cite{Z1985} and \cite{Z1987}, and they have no 
consequences for our results here.

\begin{lemma}\label{solv}
If $\Gamma$ is solvable then it is cyclic, binary dihedral, binary 
tetrahedral or binary octahedral.
\end{lemma}

\begin{proof} The odd
Sylow subgroups of $\Gamma$ are cyclic and the 2--Sylow subgroups are either
cyclic or generalized quaternionic (binary dihedral $\D^*_m$ where $m > 1$ 
is a power of $2$), because every abelian subgroup of $\Gamma$ is cyclic. 
If the $2$--Sylow
subgroups of $\Gamma$ are cyclic, we are done by the above-mentioned result
Vincent. Otherwise, $\Gamma$ has order $2^sn$ with $n$ odd and $s>2$, and an
element of order $2^{s-1}$.  From the above-mentioned result of Zassenhaus 
we have a normal subgroup $\Gamma_1 \subset \Gamma$ with all Sylow subgroups 
cyclic and $\Gamma / \Gamma_1$ equal to $\Z_2, \A_4 \text{ or } \bbS_4$\,,
and $\Gamma_1$ is either cyclic of $\D_m^*, m \text{ odd}$ by the
result of Vincent.

If $\Gamma / \Gamma_1 = \Z_2$ then $\Gamma$ is cyclic or binary dihedral 
by Lemma \ref{build}.

If $\Gamma / \Gamma_1 = \A_4$ with $\Gamma_1$ cyclic one argues that $\Gamma$
is binary dihedral or binary tetrahedral.  If $\Gamma / \Gamma_1 = \A_4$ 
with $\Gamma_1$ binary dihedral one argues that $\Gamma$ is binary dihedral.
These arguments are a little bit complicated.

If $\Gamma / \Gamma_1 = \bbS_4$ with projection $\psi: \Gamma \to \bbS_4$
then $\Gamma' := \psi^{-1}(\A_4)$ is binary dihedral or binary tetrahedral.  
If $\Gamma' = \D_q^*$ then $\Gamma = \D_{2q}^*$ by Lemma \ref{build}.  If
$\Gamma' = \T^*$ then $\Gamma_1 = \Z_2$ and $\Gamma = \O^*$.
\end{proof}

Now we need only show that if $\Gamma$ is not solvable then $\Gamma = \I^*$.
For that we use $\I^* \cong SL(2,5)$, the group of $2 \times 2$ unimodular
matrices over the field $\F_5$\,, and M. Suzuki's theorem 
\cite[Theorem E]{S1955} 
that a non--solvable group with every abelian subgroup cyclic has a normal
subgroup isomorphic to some $SL(2,p)$ with $p > 3$ prime.

\begin{lemma}\label{non-solv-1}
If $\Gamma \cong SL(2,p)$ then $p = 3$ or $p = 5$.
\end{lemma}

\begin{proof} Let $\omega$ generate the multiplicative group of non-zero 
elements of the field of $p$ elements, and 
$$
\nu = \left ( \begin{smallmatrix} \omega & 0 \\ 0 & \omega^{-1} 
	\end{smallmatrix} \right ) \text{ and } \alpha =
\left ( \begin{smallmatrix} 1 & 1 \\ 0 & 1 
	\end{smallmatrix} \right ) \text{ in } SL(2,p)
$$
Then $\nu\alpha\nu^{-1} = \alpha^{\omega^2}$ so $\omega^2 = \pm 1$ mod $p$
by Lemma \ref{clif-basics}, so $\omega^4 = 1$, so $p-1$ divides $4$.
Thus $p$ is $2$, $3$, or $5$, and $p \ne 2$ because $SL(2,2)$ 
has several éléments of order 2.
\end{proof}

\begin{lemma}\label{non-solv-2}
If $\Gamma$ has a normal subgroup $\Gamma_1 \cong SL(2,5)$ then
$\Gamma = \Gamma_1$\,.
\end{lemma}

\noindent {\sl Idea of Proof.} One argues that $\Gamma / \Gamma_1
\subset \Aut(\Gamma_1)/\Int(\Gamma_1)$.  That group has order 2,
and its nontrivial element $\Ad(\sigma)$ is represented by conjugation
by $\left ( \begin{smallmatrix} 0 & -1 \\ 2 & 0 
        \end{smallmatrix} \right )$. 

Suppose $\Gamma \ne \Gamma_1$. We can assume $\sigma^2$ central in 
$\Gamma$, so $\sigma^2 = -1 \in SL(2,5)$.  Denote $\alpha = 
\left ( \begin{smallmatrix} 1 & 1 \\ 0 & 1 
        \end{smallmatrix} \right )$, 
$\beta = \left ( \begin{smallmatrix} 1 & 0 \\ 1 & 1 
        \end{smallmatrix} \right )$
and $\gamma = \left ( \begin{smallmatrix} 0 & -1 \\ 1 & 0 
        \end{smallmatrix} \right )$ in $SL(2,5)$.
Then $\Ad(\sigma)\alpha = \beta^3$ and $\gamma\alpha\gamma^{-1} = \beta^{-1}$
so $\beta$ is conjugate in $\\Gamma$ to $\beta^{-3} = \beta^2$.  That
would say $\beta = I$ or $\beta^3 = I$, which is a contradiction.
\hfill $\diamondsuit$

\begin{lemma}\label{non-solv-3}
If $\Gamma$ is not solvable then $\Gamma \cong \I^*$.
\end{lemma}

\begin{proof} Lemmas \ref{non-solv-1} and \ref{non-solv-2}, and the 
result of Suzuki \cite[Theorem E]{S1955}, show $\Gamma \cong SL(2,5)$.  
\end{proof}

Combining Lemmas \ref{solv} and \ref{non-solv-3} we have the first part of

\begin{proposition}\label{const-curve-pos}
Consider a Riemannian manifold $(M',ds'^2) = \Gamma \backslash \bbS^n$
of constant positive curvature.  Suppose that every element 
$\gamma \in \Gamma$ is an isometry of constant displacement on $\bbS^n$. 
Then the group $\Gamma$ is cyclic, binary dihedral, 
binary tetrahedral, binary octahedral or binary icosahedral, and
$(M',ds'^2)$ is homogeneous.  
\end{proposition}

\noindent {\em Idea of Proof.}
Since every $\gamma \in \Gamma$ is of constant displacement on $\bbS^n$,
we know the structure of $\Gamma$ from Lemmas \ref{solv} and \ref{non-solv-3},
and we run through the cases to see in each case that $\Gamma$ is given
as in (\ref{div-alg}).  This is immediate for $\Gamma$ cyclic of order
$1$ or $2$.  For $\Gamma$ cyclic of order $m > 2$ a generator is given
by $J_m$ in \ref{div-alg}, so $\Gamma$ is given by Case 2 of \ref{div-alg}. 

If $\Gamma = \D_m^*$ with $m > 2$ even it has a normal subgroup
$\Gamma_1 = \Z_{2m}$ with generator $\gamma$ as in Case 2 of \ref{div-alg}
with $k = 2m$, and another generator $\beta$ with $\beta\gamma\beta^{-1}
= \gamma^{-1}$.  Then $\beta^2 = \gamma^m = -1$ so $\beta$ acts as
$\left ( \begin{smallmatrix} 0 & 1 \\ -1 & 0 \end{smallmatrix} \right )$
in block form on the two eigenspaces of $\gamma$.  Thus $\Gamma$ is
given by Case 3 of \ref{div-alg}.  

If $\Gamma = \D_m^*$ with $m > 2$ odd, the argument is similar.

If $\Gamma$ is of form $\T^*$ or then $\O^*$ one builds it up from
a cyclic subgroups as for the even binary dihedral cases.  As was
implicit there, the building blocks have exactly two joint eigenspaces
and an element of the normalizer exchanges them.

If $\Gamma = \I^*$ there are only two irreducible representations that
give isometries of constant displacement on the corresponding spheres.
They have degree 2 and are $O(4)$--conjugate, so we may assume that
$\Gamma \to U(\frac{n+1}{2})$ is a sum of copies of just one of those
irreducibles.  Thus $\Gamma$ is given by Case 3 of \ref{div-alg}.
\hfill $\diamondsuit$

Theorem \ref{hc-const-curv} follows by combining Lemmas \ref{const-curve-neg}
and \ref{const-curve-zero} with Proposition \ref{const-curve-pos}.
\hfill $\square$

\section{\bf Riemannian Symmetric Spaces with Simple Isometry Group.}
\label{sec3}
\setcounter{equation}{0}
\setcounter{subsection}{0}

In Sections \ref{sec3} and \ref{sec4} we describe our original proof the 
Homogeneity Conjecture for the cases where $(M,ds^2)$ is a Riemannian 
symmetric space.  
That uses \' E. Cartan's classification of symmetric spaces.  Some of the 
results along the way are valid more generally for Riemannian homogeneous 
spaces.  We indicate that situation in the notation as well as the statements, 
in some cases 
writing $(L,dt^2)$ and $(L',dt'^2)$ instead of $(M,ds^2)$ and $(M',ds'^2)$.
Then in Section \ref{sec5} we describe Ozols' argument which minimizes
the use of classification in the proof of Theorem \ref{symm-simple-isom}.

\begin{theorem}\label{hc-symmetric}
Let $(M,ds^2)$ be a connected simply connected Riemannian symmetric space.
Let $\pi: (M,ds^2) \to (M',ds'^2) = (\Gamma \backslash M, ds'^2)$ be a
Riemannian covering.  Then $(M',ds'^2)$ is homogeneous if and only if every
$\gamma \in \Gamma$ is an isometry of constant displacement on $(M,ds^2)$.
\end{theorem}

There are three surprises here.  One is that the divergent geodesic argument
of Lemma \ref{const-curve-neg} works (with some modification) for symmetric 
spaces
of noncompact type.  The second is that the round sphere, as in Proposition
\ref{const-curve-pos}, is the hard case for symmetric spaces of compact
type.  Perhaps that is because the sphere has so many isometries.  The
third is that the result holds for Finsler symmetric spaces \cite{DW2012}.

If $(M,ds^2)$ is a Riemannian product, $(M,ds^2) = 
(M_1,ds_1^2) \times (M_2,ds_2^2)$ with no Euclidean factor, and 
if $\gamma$ is an isometry of constant displacement on $(M,ds^2)$, then 
$\gamma = \gamma_1 \times \gamma_2$ where each $\gamma_i$
is an isometry of constant displacement on $(M_i,ds_i^2)$.  Thus, in the
de Rham decomposition
$$
(M,ds^2) = (M_0,ds_0^2) \times (M_1,ds_1^2) 
	\times \dots \times (M_\ell,ds_\ell^2),
$$
$(M_0,ds_0^2)$ euclidean and the other $(M_i,ds_i^2)$ irreducible, we have
$\gamma = \gamma_0 \times \gamma_1 \times \dots \times \gamma_\ell$ where
each $\gamma_i$ is an isometry of constant displacement on $(M_i,ds_i^2)$.
Consequently we need only prove Theorem \ref{hc-symmetric} for each of the
$\gamma_i$.  By Lemma \ref{const-curve-zero} we already know it for $\gamma_0$,
so in the proof of Theorem \ref{hc-symmetric} we may (and do) assume that 
$(M,ds^2)$ is an irreducible Riemannian symmetric space.

\medskip
\addtocounter{subsection}{1}
\centerline{\bf \thesubsection.  Symmetric Spaces of Nonpositive Curvature.}
\label{ss3a}

\begin{proposition}\label{ill-j}{\rm (\cite[Theorem 1]{W1964})} If $(L,dt^2)$
is a complete connected simply connected Riemannian manifold of sectional
curvature $\geqq 0$, with no euclidean factor in its de Rham decomposition,
then every bounded isometry of $(L,dt^2)$ is trivial.  In particular if a
Riemannian quotient $(L',dt'^2) := \Gamma \backslash (L,dt^2)$ is homogeneous 
then $\Gamma = \{1\}$ and $(L',dt'^2) = (L,dt^2)$.
\end{proposition}

The symmetric space case is a special case of Proposition \ref{ill-j}:

\begin{proposition}\label{symm-neg}
Any isometry of bounded displacement on an irreducible Riemannian symmetric 
space $(M,ds^2)$ of noncompact type is the identity transformation.  In 
particular, if $(M',ds'^2) = \Gamma \backslash (M,ds^2)$ is homogeneous 
then $(M',ds'^2) = (M,ds^2)$.
\end{proposition}

We now assume that $(M,ds^2)$ is an irreducible Riemannian symmetric space
of compact type.  We will use two results of \' Elie Cartan: the classification
and the description of the full group of isometries.  The description in
question as formulated in \cite[Theorem 8.8.1]{W1966} is

\begin{proposition}\label{full-group}
{\rm (\'E. Cartan)}  Let $(M = G/K, ds^2)$ be an irreducible simply
connected Riemannian symmetric space where $G = \bI^0(M,ds^2)$. Let $s$
be the symmetry. Define $K'' = K \cup s\cdot K$ and $G'' = G \cup s\cdot G$.
Let $\Aut(K)^G$ (resp. $\Int(K)^G$) denote the group of all (resp. all
inner) automorphisms of $K$ that extend to automorphisms of $G$.
Let $K'$ be the isotropy subgroup of $\bI(M,ds^2)$.  Then 
$K' = \bigcup (k_i\cdot K'')$ and  $\bI(M,ds^2) = \bigcup (k_i\cdot G'')$, 
disjoint unions where $\Aut(K)^G = 
\bigcup \left (\Ad(k_i)|_K\cdot \Int(K)^G \right )$ disjoint.
\end{proposition}

\smallskip
\addtocounter{subsection}{1}
\centerline{\bf \thesubsection.  Compact Simple Isometry Group Using 
	Classification.}

We run through the steps in the argument for Theorem \ref{hc-symmetric}
with $\bI(M,ds^2)$ compact and simple.

\begin{lemma}\label{eq-rank-1} {\rm (\cite[Theorem 5.2.2]{W1962})}
Let $(L,dt^2)$ be a connected homogeneous Riemannian manifold 
such that the identity components of the isotropy subgroups of 
$(L,d\tau^2)$ are irreducible on the tangent spaces.
Suppose that $\beta \in \bI(L,dt^2)$ centralizes $\bI^0(L,dt^2)$, 
$g \in \bI(L,dt^2)$ has a fixed point on $L$, and $\gamma = g\beta$ is an
isometry of constant displacement on $(L,dt^2)$.  Then $\gamma = \beta$,
i.e. $g = 1$.
\end{lemma}

\begin{lemma}\label{eq-rank-2} {\rm (\cite[Corollaries 5.2.3 and 5.2.4]{W1962})}
Suppose that $(L,dt^2)$ is a compact homogeneous Riemannian manifold 
of Euler characteristic $\chi(L) > 0$.  Suppose further that 
the identity components of the isotropy subgroups of
$(L,d\tau^2)$ are irreducible on the tangent spaces.  Finally suppose that
$\bI(L,dt^2) = \bigcup \left ( \beta_j\cdot \bI^0(L,dt^2)\right )$ 
where the $\beta_j$ centralize $\bI^0(L,dt^2)$. {\rm (}This is
automatic if $\bI^0(L,dt^2)$ has no outer automorphism, in other words if
$\bI^0(L,dt^2)$ is not of type $A_n (n > 1)$, $D_n (n > 3)$ nor $E_6$.{\rm )}  
Let $\Gamma$ be a group of isometries of constant displacement on $(L,dt^2)$.
Then $\Gamma$ centralizes $\bI^0(L,dt^2)$, so $(L',dt'^2) := \Gamma \backslash
(L,dt^2)$ is homogeneous.
\end{lemma}

Now we need a result of Jean de Siebenthal \cite[pp. 57--58]{dS1956}.
\begin{proposition}\label{dS}
Let $L^0$ be the identity component of a compact Lie group $L$, $x \in L$, 
and $T$ a maximal torus of the centralizer of $Z_L(x)$. Write
$\<T,x\>$ for the group generated by $T$ and $x$.  Then every element 
of the component $xL^0$ of $x$ is $\Ad(L^0)$--conjugate to an element of 
$xT$.  If $x' \in Lx$ and $T'$ is a maximal torus of $Z_L(x')$ then
$\<T',x'\>$ is $\Ad(L^0)$--conjugate to $\<T,x\>$.
\end{proposition}
\noindent to show that in certain components of certain $\bI(M,ds^2)$ every 
element has a fixed point, in particular those components do not contain any 
isometries of constant displacement.

\begin{lemma}\label{eq-rank-3} {\rm (\cite[Lemma 5.3.1]{W1962})} 
Let $\sigma$ be a symmetry of a compact connected Riemannian symmetric space 
$(M,ds^2)$. Then every element of $\sigma\cdot \bI^0(M,ds^2)$ has a fixed point 
on $M$.
\end{lemma}
\begin{proof}
Write $\sigma$ for the symmetry at $x_0 \in M$. Let $T$ be a maximal torus of 
the isotropy subgroup $K'$ of $\bI(M,ds^2)$ at $x_0$\,. $K'$ and the centralizer 
of $\sigma$ in $\bI(M,ds^2)$ have the same identity component. Thus 
de Siebenthal's theorem shows that every element of 
$\sigma\cdot \bI^0(M,ds^2)$ is conjugate to an 
element of $(\sigma\cdot T) \subset K'$. The Lemma follows. 
\end{proof}

The argument of Lemma \ref{eq-rank-3} proves
\begin{lemma}\label{eq-rank-4} {\rm (\cite[Lemma 5.3.2]{W1962})}
Let $(L,dt^2)$ be a compact connected Riemannian homogeneous manifold,
$K'$ the isotropy subgroup of $\bI(L,dt^2)$ at $x_0 \in L$, $k \in K'$, and 
$g \in k\cdot \bI^0(L,dt^2)$. Suppose that $K'$ contains a maximal torus 
of the centralizer of $k$  in $\bI(L,dt^2)$. Then $g$ has
a fixed point on $L$.
\end{lemma}

\begin{lemma}\label{eq-rank-5} {\rm (\cite[Lemma 5.3.3]{W1962})}
Let $(M,ds^2)$ be a compact connected irreducible Riemannian symmetric 
manifold such that a connected isotropy subgroup of $\bI(M,ds^2)$ admits no 
outer automorphism. Let $\gamma$ be an isometry of $(M,ds^2)$ which has no 
fixed point. Then $\gamma \in \bI^0(M,ds^2)$.
\end{lemma}

\begin{proof}
$\bI(M,ds^2) = \bI^0(M,ds^2) \cup \sigma \cdot \bI^0(M,ds^2)$ by Proposition 
\ref{full-group}.  Now apply Lemma \ref{eq-rank-3}.
\end{proof}

Now we come to the main result of this section.

\begin{theorem}\label{symm-simple-isom} {\rm (\cite[Theorem 5.5.1]{W1962})}
Let $(M,ds^2)$ be a compact connected simply connected irreducible
Riemannian symmetric manifold with $\bI^0(M,ds^2)$ simple. Let 
$\Gamma$ be a group of isometries of constant displacement on $(M,ds^2)$. 

If $\Gamma$ is finite, then $(M',ds'^2) = \Gamma \backslash (M,ds^2)$ is 
a Riemannian homogeneous manifold. If $(M,ds^2)$ is neither an odd dimensional 
sphere, 
nor a space $SU(2m)/Sp(m)$ with $m > 1$, nor a complex projective space of 
odd complex dimension $ > 1$, nor a space $SO(4n+2)/U(2n+1)$ with $n>0$, 
then $\Gamma$ is finite and centralizes $\bI^0(M,ds^2)$, and $(M',ds'^2)$ 
is a Riemannian symmetric manifold; otherwise $(M,ds^2)$ has
finite groups of constant displacement isometries and the corresponding
quotients are Riemannian homogeneous but not Riemannian symmetric. 

If $(M,ds^2)$ is neither an odd dimensional sphere nor a space 
$SU(2m)/Sp(m)$ with $m>1$, then $\Gamma$ is finite;
otherwise $(M,ds^2)$ has one--parameter groups of constant displacement 
isometries.
\end{theorem}

\noindent {\em Indication of Proof.}
Let $G= \bI^0(M,ds^2)$ and let $K$ be the isotropy subgroup at $x\in M$.
By  Lemma \ref{eq-rank-2} we need only check the cases where $G$ is of type 
$A_n (n >1)$, $D_n (n > 3)$ or $E_6$\,.  As the statements are known for 
spheres (see Section 2), É. Cartan's classification of symmetric spaces 
shows that we need only check the
cases 

(AI) $SU(n)/SO(n), n>2$; (AII) $SU(2n)/Sp(n), n>1$;
(AIII) $SU(p+q)/\{S(U(p) \times U(q))\}, pq>1$; 

(DI) $SO(p+q)/SO(p)\times SO(q)$ with $p \geqq 2, q \geqq 2,
p+q > 4, p+q$ even; (DIII) $S0(2n)/U(n), n>1$; 

(EI) $E_6/\Ad(C_4)$; (EII) $E_6/\{A_5\times A_1\}$; 
(EIII) $E_6/\{D_5 \times T_1\}$; (EIV) $E_6/F_4$.

\smallskip
The cases where $G$ is a classical group involve quite a lot of matrix 
calculation, and the
$E_6$ cases depend on classical subgroups of $E_6$\,.
\hfill $\diamondsuit$

\section{\bf Riemannian Symmetric Spaces that are Group Manifolds.}\label{sec4}
\setcounter{equation}{0}
\setcounter{subsection}{0}

In this section we indicate the proof of the Homogeneity Conjecture for 
compact Riemannian group manifolds with bi--invariant metric.  Those 
symmetric spaces were introduced by É. Cartan in \cite{C1927b} and led to his
development of symmetric space theory.

Fix a compact connected simply connected simple Lie group $G$ with
bi--invariant Riemannian metric $ds^2$.  Let $T(G)$ denote the set of all
isometries of the form
$$
(g_1,g_2): x \mapsto g_1^{-1}xg_2 \text{ where } x,\, g_1,\, g_2 \in G.
$$
Then $T(G)$ is the identity component of the isometry group $\bI(G,ds^2)$.

\begin{lemma}\label{transl-0}
{\rm (\cite[Lemmas 4.2.1, 4.2.2]{W1962})}
If $(g_1,g_2) \in T(G)$ is an isometry of constant displacement on
$(G,ds^2)$ then every conjugate of $g_1$ commutes with every conjugate
of $g_2$\,, so either $g_1$ or $g_2$ is central in $G$.
\end{lemma}

The symmetry $\sigma : x \mapsto x^{-1}$ satisfies 
$\sigma\cdot (g_1,g_2)\cdot\sigma^{-1} = (g_2,g_1)$.  Let $I'(G,ds^2)$
denote the group $T(G) \cup \sigma\cdot T(G)$ of isometries of $(G,ds^2)$.

\begin{proposition}\label{transl-1}
{\rm (\cite[Lemma 4.2.3, Theorem 4.2.4]{W1962})}
Every element of $\sigma\cdot T(G)$ has a fixed point on $G$.  Every
subgroup $\Gamma \subset I'(G,ds^2)$ of isometries of constant
displacement is conjugate in $I'(G,ds^2)$ to a group of left translations
of $G$.
\end{proposition}

\begin{lemma}\label{transl-2}
{\rm (\cite[Lemmas 4.3.1, 4.3.2]{W1962})}
Let $\alpha$ be an automorphism of $G$.  Let $g \in G$ and suppose that
$x \mapsto g\alpha(x)$ is an isometry of constant displacement on
$(G,ds^2)$.  Then $\alpha( ug\alpha(u^{-1}) ) = ug\alpha(u^{-1})$ for
ever $u \in G$.  Let $B$ denote the identity component of the centralizer
of $g$ in $G$. Then $\alpha(B) = B$, and if $\alpha|_B$ is an inner 
automorphism of $B$ then $\alpha$ is inner on $G$.
\end{lemma}

\begin{proposition}\label{transl-3}
{\rm (\cite[Theorem 4.3.3]{W1962})}
Let $\alpha \in \Aut(G)$, $(g_1,g_2)\in T(G)$, and $\gamma = 
(g_1,g_2)\cdot\alpha$. Suppose that both $\gamma$ and $\gamma^2$
are isometries of constant displacement on $(G,ds^2)$.  Then
$\alpha$ is an inner automorphism of $G$.
\end{proposition}

\begin{lemma}\label{transl-4}
{\rm (\cite[Lemma 4.3.4]{W1962})}
Let $\alpha \in \Aut(G)$, $\sigma$ the symmetry $x \mapsto x^{-1}$
and $g \in G$.  Suppose that $\gamma = (g,1)\cdot\alpha\cdot\sigma$
is an isometry of constant displacement on $(G,ds^2)$.  Then
$(ug \alpha^2(u^{-1}))^{-1} = \alpha(ug \alpha^2(u^{-1}))$
for every $u \in G$.
\end{lemma}

Now we come to the main results of this section:

\begin{theorem}\label{symm-group}
{\rm (\cite[Theorem 4.5.1]{W1962})}
Let $G$ be a compact connected Lie group and $ds^2$ a bi--invariant
Riemannian metric on $G$.  Let $\Gamma$ be a group of isometries of constant 
displacement on $(G,ds^2)$.  Then $\Gamma$ is conjugate in $\bI(G,ds^2)$ to 
a group of left translations on $G$.  Thus $\Gamma\backslash (G,ds^2)$ is
homogeneous.

In particular, the Homogeneity Conjecture holds for Riemannian symmetric 
spaces that are group manifolds.
\end{theorem}

Combining Theorems \ref{symm-simple-isom} and \ref{symm-group} we have
verified the Homogeneity Conjecture (\ref{hc-symmetric})
for Riemannian symmetric spaces.

\begin{remark}{\rm In Theorem \ref{symm-group} the metric $ds^2$ on $G$
is bi-invariant.  Later we will see the extent to which the result holds more
generally for left--invariant metrics on Lie groups.}
\end{remark}
\medskip

\section{\bf A Classification Free Approach.} \label{sec5}
\setcounter{equation}{0}
\setcounter{subsection}{0}
\medskip

In this section we look at the characterizations of constant displacement 
isometries in terms of invariant geodesics.  We follow the notation of
\cite{O1974}, slightly adjusted for consistency with the other sections
of this paper.  Let $(M,ds^2)$ be a complete Riemannian manifold and 
$\sigma: \R \to M$ a geodesic parameterized proportional to arc length.
Then an isometry $\gamma \in \bI(M,ds^2)$ {\em preserves} $\sigma$ if 
there  is a constant $c$ such that $\gamma(\sigma(t)) = \sigma(t+c)$ for 
all $t \in \R$.  We say that
$$
\begin{aligned}
&\gamma \text{ satisfies } P_x \text{ if } \gamma 
	\text{ preserves at least one minimizing geodesic from } x \text{ to }
	\gamma(x),\\
&Pres(\gamma) = \{x \in M \mid \gamma \text{ satisfies } P_x\}, \text{ and } \\
&Crit(\gamma) = \{x \in M \mid x \text{ is a critical point of the square of
	the displacement function of } \gamma\}.
\end{aligned}
$$
In \cite{O1969} it was shown that if $\gamma$ has small displacement, i.e.
if $\gamma(x)$ is not in the cut locus of $x$ for any $x \in M$, then
$Crit(\gamma) = Pres(\gamma)$, and from that isometries of small constant
displacement have transitive centralizers. The results of \cite{O1974} 
extend that.  The main results of this section, taken from 
\cite{O1974}, are the following.

\begin{proposition}\label{oz1}{\rm \cite[Theorem 1.6]{O1974}}
Let $(M,ds^2)$ be a complete connected $C^\infty$ Riemannian manifold.
Given an isometry $\gamma \in \bI(M,ds^2)$ the following are equivalent.

{\rm 1.} $\gamma$ is an isometry of constant displacement.

{\rm 2.} If $x \in M$ then $\gamma$ preserves some minimizing geodesic
	from $x$ to $\gamma(x)$.

{\rm 3.} If $x \in M$ then $\gamma$ preserves every minimizing geodesic
	from $x$ to $\gamma(x)$, i.e. $Pres(\gamma) = M$.
\end{proposition}

\begin{proposition}\label{oz2}{\rm \cite[Theorem 2.6]{O1974}}
Let $(M,ds^2)$ be a connected simply connected Riemannian symmetric space
of compact type.
If $\gamma \in \bI(M,ds^2)$ then 
$Pres(\gamma) = \bigcup Z_G^0(\gamma)\cdot x_i$ where 
$\{x_i\}$ is a set of representatives of the components of $Pres(\gamma)$.
\end{proposition}

\begin{corollary}\label{oz3}{\rm \cite[Corollary 2.7]{O1974}}
Let $(M,ds^2)$ be a connected simply connected Riemannian symmetric space
of compact type.
If $\gamma \in \bI(M,ds^2)$ then $\gamma$ is an isometry of constant 
displacement if and only if $Z_G^0(\gamma)$ is transitive on $M$.
\end{corollary}

\begin{corollary}\label{oz4}{\rm \cite[Corollary 2.9]{O1974}}
Let $(M,ds^2)$ be a connected simply connected Riemannian symmetric space
of compact type.  Let $\Gamma$ be a finite cyclic group of isometries of
constant displacement.  Then $\Gamma \backslash (M,ds^2)$ is homogeneous.
\end{corollary}

\section{\bf Extension to Finsler Symmetric Spaces.}\label{sec6}
\setcounter{equation}{0}
\setcounter{subsection}{0}
\medskip

In this section we indicate the steps in the proof the Finsler symmetric 
space analog of Theorem \ref{hc-symmetric}.  The result is

\begin{theorem}\label{hc-finsler}{\rm (Deng--Wolf \cite[Theorem 1.1]{DW2012})}
Let $\Gamma$ be a properly discontinuous group of isometries of a connected
simply connected globally symmetric Finsler space $(M, F)$.  Then
$\Gamma\backslash (M, F)$ is a homogeneous Finsler space if and only if
$\Gamma$ consists of isometries of constant displacement.  Further, if 
$\Gamma\backslash (M, F)$ is
homogeneous, and if in the decomposition of $(M, F)$ as the
Berwald product of Minkowski space and irreducible symmetric Finsler
spaces, none of whose factors is
\begin{itemize}
\item[] a compact Lie group with a bi-invariant Finsler metric,
\item[] an odd-dimensional sphere with the constant curvature Riemannian metric,
\item[] a complex projective of odd complex dimension $>1$ with the 
	standard Riemannian metric,
\item[] $\mathrm{SU}(2n)/\mathrm{Sp}(n)$, $n\geq 2$ with a possibly
	non--Riemannian $\mathrm{SU}(2n)$-invariant Finsler metric, nor
\item[] $\mathrm{SU}(4n+2)/\mathrm{U}(2n+1)$,   $n\geqq 1$, with a
	possibly non--Riemannian $\mathrm{SU}(4n+2)$-invariant Finsler metric, 
\end{itemize}
then $\Gamma\backslash (M, F)$ is a Finsler symmetric space.
\end{theorem}

This uses some essential background on Finsler spaces, all of which can be
found in the background expository parts of S. Deng's treatise \cite{D2012}.
Then the notion of Minkowski Lie Algebra is introduced as a convenience in
studying affine symmetric Berwald spaces, leading to a 
Finsler symmetric space variation on the Riemannian manifold de Rham 
decomposition.  That applies to affine symmetric Berwald spaces.  As in the 
Riemannian case, isometries of constant displacement decompose accordingly
as products.  That uses the ideas behind 
Ozols' description \cite{O1974} of constant displacement
isometries in terms of invariant minimizing geodesics, and Crittenden's
description \cite{C1962} of cut locus {\em vs} conjugate locus in
symmetric spaces.  

At that point the connected simply connected affine symmetric Berwald space
$(M,F)$ is a Berwald product $(M,F) = (M_0,F_0) \times
(M_1,F_1) \times \dots \times (M_r,F_r)$ with $(M_0,F_0)$ flat and the 
other $(M_i,F_i)$ irreducible, and the group $\Gamma$ decomposes
accordingly as 
$\Gamma = \Gamma_0 \times \Gamma_1 \times \dots \times \Gamma_r$\,.
The Homogeneity Conjecture is easy for $(M_0,F_0)$ and more or less
similar to the Riemannian case for $(M_i,F_i)$ of noncompact type.
For $(M_i,F_i)$ of compact type one needs the result of Szab\' o
\cite{S1981}, that if $(M, F)$ is a Berwald space, then there exists a 
Riemannian metric $Q$ on $M$ whose Levi-Civit\`a connection coincides
with the linear connection of $(M, F)$.  With that one can reduce 
the Homogeneity Conjecture for $(M_i,F_i)$ of compact type to the Riemannian
result, and the main part of Theorem \ref{hc-finsler} follows.  As in the
Riemannian case, the part on symmetric space quotients is easily extracted.

\medskip
\centerline{\bf Part II.  Geometric Classes of Compact Riemannian Manifolds.}
\medskip

In this Part we verify the Homogeneity Conjecture for several classes of
compact Riemannian homogeneous spaces.  Those are manifolds of Euler
charactristic $\chi(G/H) \ne 0$ in Section \ref{sec8}, compact group 
manifolds in Section \ref{sec9} and manifolds of positive curvature in
Section \ref{sec10}.  But first we develop some tools in Section \ref{sec7} 
that we will need for those geometrically defined classes of manifolds.

\section{\bf Isotropy Splitting Fibrations.}\label{sec7}
\setcounter{equation}{0}
\setcounter{subsection}{0}

In this section we develop and apply a tool for reducing cases of the 
Homogeneity Conjecture to Riemannian symmetric spaces and other cases that
we will have verified.  The tool is based on the idea of the classical 
fibration $SO(k + \ell)/SO(k) \to (SO(k)\times SO(\ell))$ of a Stieffel 
manifold over a Grassmann manifold.  It applies to real, complex and 
quaternionic Stieffel manifolds and will be a key to considering Riemannian 
homogeneous spaces of the classes mentioned above.

\medskip
\addtocounter{subsection}{1}
\centerline{\bf \thesubsection.  Definition and Goal.}
\medskip

\noindent Our basic setup here is 

\begin{equation}\label{split-def}
\begin{aligned}
&G \text{ is a compact connected simply connected Lie group,}\\
&K = HN \text{ where } H \text{ and } N \text{ are closed connected 
	subgroups of } G \text{ such that} \\
&\qquad {\rm (i)}\,\, K = (H \times N)/(H \cap N),\, 
        {\rm (ii)}\,\, \text{Lie algebras }\gh \perp \gn \text{ and } 
        {\rm (iii)}\,\, \dim H \ne 0 \ne \dim N\,,\\
&\text{the centralizers } Z_G(H) = Z_{H} N^\flat
        \text{ and } Z_G(N) = H^\flat Z_{N} \text{ with }
        H = (H^\flat)^0\, \text{ and } N = (N^\flat)^0 \\
&M = G/H \text{ and } M'' = G/K \text{ are normal Riemannian
        homogeneous spaces of }G.
\end{aligned}
\end{equation}

We may assume that the the metrics $ds^2$ on $M$ and $ds''^2$ on 
$M''$ are the normal
metrics given by the negative of the Killing form of $\gg$.  Also, since $G$
is simply connected, and $H$ and $N$ are connected, $M$ and $M''$ are
simply connected.  We refer to the fibration $\pi:  M \to M''$
as an {\em isotropy splitting fibration}.  
Since $N$ normalizes $H$, $G \times N$ acts on $M$ by left and 
right translations, $\ell(g)r(n)(xH)
= gxHn^{-1} = gxn^{-1}H$.  The main result on isotropy splitting fibrations is

\begin{theorem}\label{cw}
Suppose that $M = G/N$ is compact, connected and simply connected, 
that $\rank K = \rank G$ and that $\{G,H,N\}$ satisfies {\rm (\ref{split-def})}.
If $\Gamma$ is a group of isometries of constant displacement on
$(M,ds^2)$ then $\Gamma \subset (\ell(Z_G) \times r(N))$ where 
$Z_G$ denotes the center of $G$.  Conversely, if $\Gamma \subset 
(\ell(Z_G) \times r(N))$ then every $\gamma \in \Gamma$ is an isometry
of constant displacement on $(M,ds^2)$.
\end{theorem}

\begin{examples}\label{examples}
Before indicating the proof of {\rm Theorem \ref{cw}} we mention some
interesting examples.
{\rm 
\begin{itemize}
\item[{\rm (1)}] $G/K$ is an irreducible hermitian symmetric space and 
	$K = Z_K^0K'$.
   \begin{itemize}
      \item $G/K' \to  G/K$ is a circle bundle over $G/K$\qquad\qquad\qquad 
		$(H = K'$ and $N = Z_K^0)$
      \item $G/Z_K^0 \to G/K$ is a principal $K'$--bundle over $G/K$\qquad\,
		$(H = Z_K^0$ and $N = K'$)
      \item $SU(s + t)/SU(s) \to SU(s + t)/S(U(s)U(t))$ and 
		$SU(s + t)/U(s) \to SU(s + t)/S(U(s)U(t))$
   \end{itemize}
\item[{\rm (2)}] $G/K$ is a quaternion–kaehler symmetric space and 
	$K = Sp(1)K'$.  
   \begin{itemize}
      \item $G/K = SO(s+t)/SO(s)SO(t)$ with $s = 3, 4$, $s \leqq t$, $st$ even
      \item $G/K = Sp(s+t)/Sp(s)Sp(t)$ with $1 \leqq s \leqq t$
      \item $G/K = G_2/A_1A_1\,, F_4/A_1C_3\,, E_6/A_1A_5\,, E_7/A_1D_6
		\text{ or } E_8/A_1E_7$
   \end{itemize}
\item[{\rm (3)}] $G/K$ is a nearly--kaehler $3$--symmetric space and 
$K = SU(3)K'$ so
$G/K$ is $G_2/A_2$, $F_4/A_2A_2$, $E_6/A_2A_2A_2$, $E_7/A_2A_5$, or 
		$E_8/A_2E_6$\,.  In the $F_4$ case the $A_2 = SU(3)$
		can be given by a long root or a short root. The corresponding 
		fibrations are $G/K'\to G/K$ principal $SU(3)$ bundle and 
		$G/SU(3) \to  G/K$ principal $K'$ bundle.
\item[{\rm (4)}] $G/K$ is the $5$--symmetric space $E_8/A_4A_4$, yielding to two 
	principal $SU(5)$ bundles $E_8/A_4 \to E_8/A_4A_4$\,.
\item[{\rm (5)}] $G/K$ is an odd dimensional real Grassmann manifold and
$K = SO(2s+1)\times SO(2t+1)$.  Then $G/K_1 \to G/K$ is 
	$SO(2s+2t+2)/SO(2s+1) \to SO(2s+2t+2)/(SO(2s+1)\times SO(2t+1))$.
\end{itemize}
}
\end{examples}

\smallskip
\addtocounter{subsection}{1}
\centerline{\bf \thesubsection.  The Full Isometry Group.}
\medskip

\noindent
Here we run through the arguments of \cite{W2018}.  The corresponding
part of Section \ref{sec9} will correspond to a limiting case $\gn = 0$.
As we saw dealing with compact symmetric spaces, verification of the 
Homogeneity Conjecture requires that we find the full group
of isometries of $(M,ds^2)$.  We do that now.  As more or less noted earlier,

\begin{lemma}\label{rt-act}
The right action of $N^\flat$ on $M$, given by 
$r(n)(gH) = gHn^{-1} = gn^{-1}H$, is a well defined
action by isometries. The fiber of $\pi: M \to M''$ through 
$gH$ is $F:= r(N)(gH)$.
\end{lemma}

Now we have larger (than $G$) transitive groups of isometries of
$M$ given by
\begin{equation}\label{g-flat}
G^\flat = \ell(G) \cdot r(N^\flat) \text{ and }
(G^\flat)^0 = \ell(G) \cdot r(L) \text{ acting by }
(\ell(g), r(n)) : xH \to g(xH)n^{-1} = gxn^{-1}H\,.
\end{equation}
Every $g^\flat := (\ell(g), r(n))$ sends fiber to fiber and induces 
an isometry $\ell(g) \in (M'',ds''^2)$.  Specializing a theorem of Reggiani 
\cite[Corollary 1.3]{R2010} for the first assertion we have

\begin{proposition}\label{fullgroup1}
If the Riemannian manifold $(M, ds^2),\, M = G/H$, is irreducible for its
de Rham decomposition, then 
$(G^\flat)^0$ is the identity component $\bI^0(M,ds^2)$ of 
its isometry group.  Consequently

{\rm (1)} The algebra of all Killing vector fields on $(M,ds^2)$ is 
$d((\ell,r)(\gg^\flat)) := d\ell(\gg) \oplus dr(\gn)$.

{\rm (2)} Define $\gt = \{\xi \in \gg \mid \xi \text{ defines a constant length 
Killing vector field on }(M'',ds''^2)\}$.  Then the set of all constant length 
Killing vector fields on $(M,ds^2)$ is $d\ell(\gt) \oplus dr(\gn)$.

{\rm (3)} If  $\rank K = \rank G$ then $\gt = 0$ and $dr(\gn)$ is the set of 
all constant length Killing vector fields on $(M,ds^2)$.
\end{proposition}

\begin{corollary}\label{norm-k2}
If $H \not\cong N$ then
every isometry of $(M,ds^2)$ normalizes $r(N)$ and sends fiber to
fiber of $\pi: M \to M''$.
\end{corollary}

\noindent
The normalizer of $H$ in $G$ also normalizes the centralizer of $N$
so it normalizes $K$ as well. Denote
\begin{equation}\label{outer}
\begin{aligned}
& \Aut(G,H) = \bigl \{\alpha \in \Aut(G) \mid \alpha(H) = H\},\,
        \Int(G,H) = \{\alpha \in \Aut(G,H) \mid \alpha|_H \text{ is inner}\},\\
&       \phantom{\Aut(G,H,ds^2) =}      \text{ and }
        \Out(G,H) = \Aut(G,H)/\Int(G,H);\\
& \Aut(G,H,ds^2) = \{\alpha \in \Aut(G,H) \mid \alpha^*ds^2 = ds^2\},\,
        \Int(G,H,ds^2) = \Aut(G,H,ds^2) \cap \Int(G,H),\\
&       \phantom{\Aut(G,H,ds^2) =}
        \text{ and } \Out(G,H,ds^2) = \Aut(G,H,ds^2)/\Int(G,H,ds^2),\\
\end{aligned}
\end{equation}
and similarly for $\Aut(G,K), \Int(G,K),\Out(G,K) \text{ and } xxx(G,K,ds''^2)$.
$\Out(G,H) \subset \Out(G,K)$ because $H$ is a local direct factor of $K$.
In many cases $\Out(G,H) = \Out(G,K)$ because $\gn$ is the 
$\gg$--centralizer of $\gh$ and $\gh \not \cong
\gn$.  But there are exceptions, such as orthocomplementation (which
exchanges the two factors of $K$) in the cases of Stieffel manifold fibrations 
$$
\begin{aligned}
&SO(2k)/SO(k) \to SO(2k)/[SO(k)\times SO(k)], \\
&SU(2k)/U(k) \to SU(2k)/S(U(k) \times U(k)) \text{ and } \\
&Sp(2k)/Sp(k) \to Sp(2k)/[Sp(k)\times Sp(k)].
\end{aligned}
$$ 
There are other exceptions, including $E_6/[A_2A_2A_2]$, but neither 
$F_4/A_2A_2$ nor $E_8/A_4A_4$ is an exception.

\begin{lemma}\label{old-inn-out}
Suppose that $\rank K = \rank G$.  Let $\alpha \in \Aut(G,H)$
$($and thus also $\alpha \in \Aut(G,N)$, so $\alpha \in \Aut(G,K))$. 
Then the following conditions
are equivalent: 
{\rm (i)} $\alpha|_K$ is an inner automorphism of $K$, 
{\rm (ii)} as an isometry, $\alpha \in \bI^0(M,ds^2)$, and 
{\rm (iii)} as an isometry, $\alpha \in \bI^0(M'',ds''^2)$.
\end{lemma}

Denote base points 
$
x_0 = 1H \in G/H = M \text{ and } x_0'' = 1K \in G/K = M''.
$
We reformulate Lemma \ref{old-inn-out} as

\begin{lemma}\label{old-isotropy}
Let $\rank K = \rank G$.  Recall that $ds^2$ is the normal metric on $M = G/H$
defined by the negative of the Killing form of $\gg$.
The isotropy subgroup $\bI(M,ds^2)_{x_0}$ has identity component 
$$
\bI^0(M,ds^2)_{x_0} = 
H\cdot\{(\ell(n),r(n)) \in G^\flat \mid n \in N\}
$$ 
and
\begin{equation}\label{tildes}
\bI(M,ds^2)_{x_0} = {\bigcup}_{\alpha \in \Out(G,H)}\, 
	H\alpha\cdot \bigl \{(\ell(n),r(n)) \in G^\flat \mid n 
		\in N^\flat \bigr \}
	\cong {\bigcup}_{\alpha \in \Out(G,H)}\, H\alpha\cdot N^\flat .
\end{equation}
Given $\alpha \in \Out(G,H)$ 
the component $H\alpha\cdot N^\flat = H\alpha'\cdot N^\flat $
if and only if $\alpha = \alpha'$ modulo inner automorphisms.
\end{lemma}

We now define two subgroups 
$G^\dagger \subset \bI(M,ds^2)$ and
$G''^\dagger \subset \bI(M'',ds''^2)$ 
of the isometry groups by
\begin{equation}\label{old-daggers}
G^\dagger = {\bigcup}_{\alpha \in \Out(G,H) }\, 
	G \alpha \cdot N^\flat \subset \bI(M,ds^2)
\text{ and } G''^\dagger = 
{\bigcup}_{\beta \in \Out(G,K)}\, 
G\beta \subset \bI(M'',ds^2)\,.
\end{equation}
Here $(g\alpha, n)$ acts on $M$ by 
$xH \mapsto  g\alpha(x) n^{-1}H$ and
$g\beta$ acts on $M''$ by $xK \mapsto g\beta(x)K$.
\medskip

\begin{theorem}\label{old-isogroup}
Let $\pi: M \to M''$, be an isotropy--split fibration with $M = G/H$ and
$M'' = G/K$ as in {\rm (\ref{split-def})}, and $\rank K = \rank G$.  
Recall that $ds^2$ is the normal Riemannian metric on $M$ defined by the 
negative of the Killing form of $\gg$.  Then
the identity component $\bI^0(M,ds^2) = G^\flat$ and 
the full isometry group $\bI(M,ds^2) = G^\dagger$.
\end{theorem}

\smallskip
\addtocounter{subsection}{1}
\centerline{\bf \thesubsection. Isometries of Constant Displacement.} 
\medskip
\noindent
We first accumulate a few simple observations:

(i) If $\rank K = \rank G$ and $g^\flat = (g,r(n)) \in G^\flat$
there is a $g^\flat$--invariant fiber $xF = \pi^{-1}(xK)$ of $M \to M''$,

(ii) the metrics $ds^2$ on $M = G/H$ and $ds''^2$ on $M'' = G/K$ 
are naturally reductive relative to $G$,

(iii) the isotropy–split manifold $(M,ds^2)$ is a geodesic orbit space,

(iv) the fiber $F$ of $\pi: M \to M''$ is totally geodesic in $(M,ds^2)$,

(v) $(F,ds^2|_F)$ is a geodesic orbit space.

\smallskip\noindent
The splitting fibration construction,
starting with Proposition \ref{cw-conn} just below, is used for
a flat rectangle argument:  $\xi_1$ and $\xi_2$ are
commuting Killing vector fields, typically
$\xi_2 \in dr(\gn)$ and $\xi_1 \in \gg$, such that $\xi_1 \perp \gn$ and 
both $g\times r(l) = \exp(\xi_1 + \xi_2)$ and $r(l) = \exp(\xi_2)$
have the same constant displacement.  Then the 
$\exp(t_1\xi_1 + t_2\xi_2)(1H)$, for $0 \leqq t_ \leqq 1$, 
form a flat rectangle.  There $r(l)$ is displacement along one side while
$g\times r(l)$ is displacement along the diagonal.  Since these displacements
are the same one argues that $\xi_1 = 0$.  The result is

\begin{proposition}\label{cw-conn}
Suppose that $\rank K = \rank G$.
Let $\Gamma$ be a subgroup of $G^\flat$ such that every $\gamma \in \Gamma$ is 
an isometry of constant displacement on $(M,ds^2)$.
Then $\Gamma \subset (Z_G \times r(N^\flat))$ where $Z_G$ is the center
of $G$.
\end{proposition}

Proposition \ref{cw-conn} applies, in particular, to Examples \ref{examples}(1)
through Examples \ref{examples}(4).  Now we look for the components of
the full isometry group.

\begin{lemma}\label{outers} {\rm \cite[Lemma 5.5]{W2018}}
Suppose that $\rank K = \rank G$.
Let $\alpha \in \Out(G,H)$ and $\gamma \in G^\flat \alpha$ such that
both $\gamma$ and $\gamma^2$ are isometries of constant displacement on
$(M,ds^2)$.  Then $\alpha|_H$ is an inner automorphism of $H$
and $\gamma \in (Z_G \times r(N^\flat))$.
\end{lemma}

\begin{theorem}\label{cw-cover}
Let $\pi: (M,ds^2) \to (M'',ds''^2)$, be an isotropy--split fibration with 
$M = G/H$ and $M'' = G/K$ as in {\rm (\ref{split-def})}, and 
$\rank K = \rank G$.  Recall that $ds^2$ is the normal Riemannian 
metric on $M$ defined by the negative of the Killing form of $\gg$. 
Let $(M,ds^2) \to \Gamma\backslash (M,ds^2)$ be a Riemannian
covering whose deck transformation group $\Gamma$ consists of isometries
of constant displacement.  Then $\Gamma \subset (Z_G \times r(N^\flat))$ and
$\Gamma\backslash (M,ds^2)$ is homogeneous.  In other words the Homogeneity 
Conjecture is verified for $(M,ds^2)$.
\end{theorem}

\smallskip
\addtocounter{subsection}{1}
\centerline{\bf \thesubsection.  Odd Real Stieffel Manifolds.}

\noindent
The idea of isotropy split fibrations came from the fibrations of 
Stieffel manifolds over Grassmann manifolds.  The cases $\chi(M'') \ne 0$
are covered by Theorem \ref{cw-cover}.  That leaves just one case,
Example \ref{examples}(5), the case of the
odd dimensional real Grassmann manifold where $\chi(M'') = 0$:
\begin{equation}\label{odd-fibr}
M = SO(2s+2t+2)/SO(2s+1) \to SO(2s+2t+2)/(SO(2s+1)\times SO(2t+1)) = M''.
\end{equation}
We know $\bI^0(M,ds^2) = G \times r(N)$ by Proposition \ref{fullgroup1}.
The following variation on Proposition \ref{cw-conn} uses an argument
of C\' ampoli \cite{C1986}.

\begin{proposition}\label{cw-less}
Let $\pi: M \to M''$ as in {\rm (\ref{split-def})} with
$\chi(M) = 0$.  If $\Gamma$ is a group of isometries of constant 
displacement on $(M,ds^2)$, and if $\Gamma \subset \bI^0(M,ds^2)$,
then $\Gamma \subset (Z_G \times r(N))$.  Conversely if 
$\Gamma \subset (Z_G \times r(N))$ then every
$\gamma \in \Gamma$ is an isometry of constant displacement on $(M,ds^2)$.
\end{proposition}

A look at the relevant group structure shows that
$\bI(M,ds^2) =
\bI^0(M,ds^2) \cup \sigma\cdot \bI^0(M,ds^2)$
where $\sigma$ is the symmetry on the symmetric space $(M'',ds''^2)$.  A short
matrix calculation shows that every element of $\sigma\cdot \bI^0(M,ds^2)$
has a fixed point on $M$.  In view of Theorem \ref{cw-less} we can adapt
the $\chi(M) > 0$ argument to the odd Stieffel manifold split fibration,
as follows.

\begin{proposition}\label{cw-st} 
Let $\pi: (M,ds^2) \to (M'',ds''^2)$ as in {\rm (\ref{odd-fibr})}.  If $\Gamma$
is a group of isometries of constant displacement on the Stieffel manifold
$(M,ds^2)$ then $\Gamma \subset (\{\pm I\} \times r(SO(1+2t)))$,
and $(M,ds^2) = \Gamma \backslash (M,ds^2)$ is homogeneous.
Thus the Homogeneity Conjecture holds for the Stieffel manifold fibration
{\rm (\ref{odd-fibr})} over an odd dimensional real Grassmann manifold.
\end{proposition} 

Since the odd dimensional real Grassmann manifolds are the only irreducible
Riemannian symmetric spaces $M = G/K$, $G$ simple, that fit the decomposition
$K = HN$ of (\ref{split-def}), Proposition \ref{cw-st} can be made to
appear to be more general:
\begin{quote}
Let $\pi: (M,ds^2) \to (M'',ds''^2)$ as in {\rm (\ref{split-def})} where
$M'' = G/K$, $G$ simple, is a Riemannian symmetric space.  Then
the Homogeneity Conjecture holds for $(M,ds^2)$.
\end{quote}
This will become more interesting when we look at the group 
manifold case.

\section{\bf Manifolds of Positive Euler Characteristic.}\label{sec8}
\setcounter{equation}{0}
\setcounter{subsection}{0}
In this section we describe the current state of research on the Homogeneity 
Conjecture for manifolds $(M,ds^2)$ of nonzero Euler characteristic.  If 
$M = G/K$ and the metric $ds^2$ is normal relative to $G$ then the 
Homogeneity Conjecture was verified in Theorem \ref{cw-cover}.  But if
$ds^2$ is not normal there are some open problems.  The current result
for $ds^2$ that need not be normal is as follows.

\begin{theorem}\label{euler-pos}
Suppose that $M = G/K$ is compact, connected and simply connected,
with $\chi(M) > 0$ i.e. $\rank K = \rank G$.  Suppose further that 
$G$ is simple, that $ds^2$ is a $G$--invariant Riemannian metric on $M$
{\rm (not necessarily normal)}, and that $G = \bI^0(M,ds^2)$.

Let $\Gamma$ be a group of isometries of constant
displacement on $(M,ds^2)$, such that each $\Ad(\gamma)$, $\gamma \in \Gamma$, 
is an inner automorphism of $G$.
{\rm \{This is automatic unless $G$ is of type $A_n$ ($n \geqq 2)$,
$D_n$ ($n \geqq 4$) or $E_6$.\}}.
Then $\Gamma$ centralizes $G$, each component of $\bI(M,ds^2)$ contains
at most one element of $\Gamma$, and the Riemannian quotient manifold
$\Gamma \backslash (M,ds^2)$ is homogeneous.
Conversely, of course, if\, $\Gamma \backslash (M,ds^2)$ is homogeneous
then every $\gamma \in \Gamma$ is an isometry of constant
displacement on $(M,ds^2)$.
\end{theorem}

Theorem \ref{euler-pos} will be an immediate consequence of Corollary 
\ref{no-outer} and Proposition \ref{chi-outer}.  At this time it is not known 
whether we really need the restriction that each $\Ad(\gamma)$, 
$\gamma \in \Gamma$, be an inner automorphism of $G$.  Following 
Lemma \ref{outers} that was automatic for $ds^2$ normal.

First we describe the isometry group $\bI(M,ds^2)$ where
$M = G/K$, $\rank G = \rank K$, as described just above.  The description of 
$\bI(M,ds^2)$ is of the form $G^\dagger = \bigcup Gk$, with isotropy 
subgroup $K^\dagger = \bigcup Kk$, with $k$ specified by (\ref{outer})
and (\ref{daggers}). below.  Our description
is inspired by Cartan's description (\cite{C1927a}, \cite{C1927b}) 
of the full isometry group of a symmetric space --- or see 
\cite{W1962} or \cite{W1966}.

\medskip
\addtocounter{subsection}{1}
\centerline{\bf \thesubsection.  Basic Setup.}
\smallskip

The first step in describing the isometry group $\bI(G/H,ds^2)$ 
is to have an idea of the structure of the maximal possible group
$\bI^0(G/H,dt^2)$, especially when $dt^2$ is the normal Riemannian
metric.  Then one works out $G = \bI^0(G/H,ds^2)$ from the specific
structure of $\bI^0(G/H,dt^2)$.  This is especially useful when
$\chi(M) > 0$, when $(G/H,ds^2)$ is a group manifold with simply
transitive group $G$, and when $(G/H,ds^2)$ has strictly positive curvature.
So it is convenient to introduce the definition: 
\begin{equation}\label{mx-condition}
\text{if $G = \bI^0(G/H,ds^2)$ then $G$ is {\em isometry-maximal} for 
$\bI(G/H,ds^2$)}.
\end{equation}

In practice it is not so difficult to check (\ref{mx-condition}).
For example let $M$ be the group manifold $SU(2) = Sp(1) = S^3$  and let
$\{\omega_1, \omega_2, \omega_3\}$ be the left invariant Maurer--Cartan
forms.  The left invariant metrics $ds^2$, leading to isomorphism classes
of isometry groups, are represented by
\begin{itemize}
\item[(1)] $ds^2 = \omega_1^2 + \omega_2^2 + \omega_3^2$ for
$\bI(S^3,ds^2) = O(4)$, constant curvature,

\item[(2)] $ds^2 = \omega_1^2 + \omega_2^2 +a\omega_3^2$ with
$0 < a < 1$ for $\bI(S^3,ds^2) = \{SU(2)\times U(1)\}/\{\pm (1,1)\}$, and

\item[(3)] $ds^2 = \omega_1^2 + b\omega_2^2 +a\omega_3^2$ with $0 < a < b < 1$
for $\bI(S^3,ds^2) \{SU(2)\times [\pm (1,1)]\}/\{\pm (1,1)\}$.
\end{itemize}
A less trivial example: $Sp(n)$ is transitive on the complex projective space
$P^{2n-1}(\C) = SU(2n)/U(2n-1)$, so if $ds^2$ is the Fubini--Study metric on 
$P^{2n-1}(\C)$ then $Sp(n)$ is not isometry--maximal.
But if $ds^2$ is not the Fubini--Study metric
then $Sp(n)$ is isometry-maximal for $(Sp(n)/Sp(n-1)U(1), ds^2)$.
Two other classical cases are from $SO(2r+2)/U(r+1) = SO(2r+1)/U(r)$ and
$SO(7)/(SO(5)\cdot SO(2)) = G_2 /U(2)$.
See Onischik's paper \cite{O1962}, especially Table 7, for a more
comprehensive list.  Or see \cite[Part II, Section 4]{GOV1997}.

The trivial $Sp(1)$ example just above, shows that if $ds^2$ is not the
normal metric on $G/K$ then there can be inner automorphisms of $G$ that
act on $G/K$ but do not preserve $ds^2$.  

\medskip
\addtocounter{subsection}{1}
\centerline{\bf \thesubsection. The Isometry Group} 
\smallskip

We start with the base point $x_0 = 1K \in G/K = M$.  Then as in
(\ref{outer}) we have $\Aut(G,K)$, $\Int(G,K)$ and $\Out(G,K)$,
and $\Aut(G,K,ds^2)$, $\Int(G,K,ds^2)$ and $\Out(G,K,ds^2)$.
Consider the subgroups
\begin{equation}\label{inners}
G' = \bigcup \{Gk \mid k \in \bI(M,ds^2)_{x_0} \text{ and } \Ad(k)|_K 
        \text{ is inner}\} \text{ and } K' = G' \cap \bI(M,ds^2)_{x_0}\,,
	\text{ and }
\end{equation}
\begin{equation}\label{inners2}
G'' = \bigcup \{Gk \mid k \in \bI(M,ds^2)_{x_0} \text{ and } \Ad(k)
        \text{ is inner on } G\} \text{ and } 
	K'' = G'' \cap \bI(M,ds^2)_{x_0}\,,
\end{equation}
\begin{equation}\label{daggers}
G^\dagger = \left ({\bigcup}_{\Ad(k) \in \Out(G,K,ds^2)}\, 
        G'k\right ) \subset \bI(M,ds^2) \text{ and }
K^\dagger = G^\dagger \cap \bI(M,ds^2)_{x_0}\,.
\end{equation}
Here $gk$ acts on $M$ by $xK \mapsto gkxK = g\,kxk^{-1}K$.

Here are some obvious key remarks that will apply to $G$, $G'$, $G''$ and
$G^\dagger$.
\begin{lemma}\label{cent-inner-G}
Suppose $\chi(M) > 0$ and $k \in \bI(M,ds^2)_{x_0}$\,. If $\Ad(k)$ is inner
on $\bI^0(M,ds^2)_{x_0}$, then it is inner on $\bI^0(M,ds^2)$.  In other
words, if $\Ad(k)$ is outer on $\bI^0(M,ds^2)$ it is outer
on $\bI^0(M,ds^2)_{x_0}$\,. If
$\Ad(k)$ is inner on $\bI^0(M,ds^2)$ then there exists
$\gamma \in \bI^0(M,ds^2)k$ that centralizes $\bI^0(M,ds^2)$.
\end{lemma}

\noindent \{Remark.  The case $S^{2n} = G/K = SO(2n+1)/SO(2n)$, where
$\bI(M,ds^2) = O(2n+1)$, and $k$ the diagonal matrix  
$\diag\{-1,-1,+1,\dots , +1\}$, 
shows that one can have $\Ad(k)$ inner on $G$ but outer on $K$.\}

\begin{proof} If $\Ad(k)$ is inner on $\bI^0(M,ds^2)_{x_0}$ then its
centralizer there contains a torus $T'$ that is a maximal torus in
$\bI^0(M,ds^2)$, so $\Ad(k)$ is inner on $\bI^0(M,ds^2)$.  If
$\Ad(k)$ is inner on $\bI^0(M,ds^2)$ we have $\gamma = gk \in
\bI^0(M,ds^2)k$ that centralizes $\bI^0(M,ds^2)$.
\end{proof}

Finally, we have a modified form of \cite[Theorem 3.12]{W2018}, again
following Section 4 there, as follows.

\begin{theorem}\label{isogroup}
Suppose that $\chi(M) > 0$ and $G = \bI^0(M,ds^2)$.
Then $G^\dagger$ is the full isometry group $\bI(M,ds^2)$, and the
respective isotropy subgroups of $G$ and $G^\dagger$ are $K$ and $K^\dagger$.
\end{theorem}

\smallskip
\addtocounter{subsection}{1}
\centerline{\bf \thesubsection. Inner Automorphisms and Constant 
Displacement Isometries.}

\begin{lemma}\label{halfspace}
Let $K$ be a compact connected Lie group of linear transformations on a 
real vector space $V$.  Suppose that the representation of $K$ on $V$
does not have any trivial summands.  Let $0 \ne w \in V$.  Then the orbit
$K(w)$ is not contained in a half--space.
\end{lemma}

\begin{proof}  Suppose that $K(w) \subset U$ where $U$ is a half space in 
$V$.  So $K(w) \subset  U = \{v \in V \mid \< u,v \> < 0\}$ for some 
nonzero $u \in V$.
The center of gravity $\bar{w} := \int_K k(w)dk \in U$ and $K(\bar{w}) = 
\bar{w}$.  Now $\bar{w}\R$ defines a trivial summand for the  
representation of $K$ on $V$.
\end{proof}

Now, with only the obvious changes, we have a very slight extension to
\cite[Theorem 5.2.2]{W1962}:
\begin{proposition}\label{is-central}
Let $(M,ds^2)$ be a connected Riemannian homogeneous manifold and suppose that:

{\rm (a)} The representations of the connected linear isotropy subgroups 
$\bI(M,ds^2)_x$
on the tangent spaces $T_x(M)$ satisfy the conditions on $K$
and $V$ in {\rm Lemma \ref{halfspace}}.

{\rm (b)} Suppose that $\beta \in \bI(M,ds^2)$ centralizes $\bI^0(M,ds^2)$, 
$g \in \bI(M,ds^2)$ has a fixed point on $M$, and $\gamma = g\beta$ is an 
isometry of constant displacement on $(M,ds^2)$.

Then $\gamma = \beta$, i.e. $g = 1$.
\end{proposition}

In order to apply Proposition \ref{is-central} we need to know that the action
of $\bI^0(M,ds^2)_x$ on $T_x(M)$ has no trivial summands.  For $\chi(M) > 0$
that is implicit in the root space calculations of \cite{BdS1949}.  Thus

\begin{corollary}\label{no-outer}
Let $(M,ds^2)$ be a connected Riemannian homogeneous manifold of
Euler characteristic $\chi(M) > 0$.  Then $G'' = G \times Z$ where 
$Z$ is the centralizer of $G$ in $G^\dagger$.  Specifically, 
$G'' = \bigcup G\beta_j$ with $Z = \{\beta_j\}$.  If $\Gamma$ is a group
of isometries of constant displacement on $(M,ds^2)$, and if 
$\Gamma \subset G''$,
then $\Gamma \subset Z$ and $\Gamma \backslash (M,ds^2)$ is homogeneous.
\end{corollary} 

Among other things, Corollary \ref{no-outer} says that if two isometries
$\gamma_1$ and $\gamma_2$ of constant displacement belong to the same 
component of the isometry group, so $\gamma_1\gamma_2^{-1} \in G$,
then $\gamma_1 = \gamma_2$ because the center of $G$ is reduced to the 
identity.

We complete the proof of Theorem \ref{euler-pos} as follows.  For clarity of
exposition we repeat the key statements.

\begin{proposition}\label{chi-outer}
Suppose that $G^\dagger = G''$, for example that $G$ is not of type 
$A_n$ ($n \geqq 2$), $D_n$ ($n \geqq 4$) nor $E_6$\,.  Then every 
component $Gk$ of $G^\dagger$ contains exactly one isometry of
constant displacement of $(M,ds^2)$.  In particular, the Homogeneity
Conjecture holds for $(M,ds^2)$.
\end{proposition}

\medskip
\addtocounter{subsection}{1}
\centerline{\bf \thesubsection. Outer Automorphisms and Constant 
Displacement Isometries}
\smallskip

It is an open problem to drop the outer automorphism condition
$\Gamma \subset G''$ in Theorem \ref{euler-pos} and Proposition 
\ref{chi-outer}.  It only applies to the cases where $G$ of type $A_n$
($n \geqq 2)$, $D_n$ ($n \geqq 4$) or $E_6$\,.  In the Riemannian
symmetric space case \cite{W1962} those cases were addressed by listing the
possibilities for $G/K$ and doing explicit calculations for each of them.
But there are too many nonsymmetric cases, so we need a new idea.
In brief, there are no definitive results yet.  I won't try to describe 
the current fragmentary results because they are part of an active
(at least on my part) research effort, and anything I write will quickly 
be obsolete.

\section{\bf Compact Group Manifolds.}\label{sec9}
\setcounter{equation}{0}
\setcounter{subsection}{0}

In this section we consider the Homogeneity Conjecture for Riemannian
manifolds $(M,ds^2)$ on which there is a compact simply transitive group $G$ of
isometries.  We already saw this for symmetric spaces, where $\bI^0(M,ds^2)$
is of the form $(G \times G)/\{\diag(G)\}$.  There $M$ was identified with
the group manifold $G$, and $(g_1,g_2) \in (G \times G$) acted by
$\ell(g_1)r(g_2):x \mapsto g_1xg_2^{-1}$.  The isotropy subgroup 
$\diag(G)$ of $\bI(G,ds^2)$ all inner automorphisms of $G$, possibly 
also some outer automorphisms, and conjugation 
$(g_1,g_2) \mapsto (g_2,g_1)$ by the symmetry at the identity.  We take that 
as a model for isometry groups of group manifolds, and we apply the result 
to finite groups of isometries of constant displacement.  Everything breaks 
up as a product under the decomposition of $G$ as a product of simple 
Lie groups, so we may assume that $G$ is simple.  The result is not complete;
at this time it is

\begin{theorem}\label{result-grp}
Let $(M,du^2)$ be a connected simply connected Riemannian manifold
on which a compact simple Lie group $G$ acts simply transitively by 
isometries.  Then there is a $G$--equivariant isometry 
$(M,du^2) \cong (G,ds^2)$ where the Riemannian metric $ds^2$ is 
invariant under left translations $\ell(g):x \mapsto gx$.  Let $\Gamma$
be a group of isometries of constant displacement on $(G,ds^2)$, and 
let $r(G,ds^2)$ be the group of all right translations 
$r(g'):x \mapsto xg'^{-1}$ that are isometries.  Then either 
$\Gamma \subset \ell(G)\ell(A)$ or
$\Gamma \subset r(G,ds^2)r(A)$, where $A$ is a finite set of isometries
$a$ such that the $\Ad(a)$ define outer automorphisms of $G$.  
If $\Gamma \subset r(G,ds^2)r(A)$ then
the centralizer of $\Gamma$ in $\bI(G,ds^2)$ contains $\ell(G)$, so the
Homogeneity Conjecture holds for $\Gamma$.
\end{theorem}

See more details on $A$ in Proposition \ref{lr4} at the end of 
Section \ref{sec9}.  One would like a more precise analysis of the case 
$\Gamma \subset \ell(G)\ell(A)$ to test the Homogeneity Conjecture there.

The result $(M,du^2) \cong (G,ds^2)$ is due to Ozeki 
\cite[Theorems 2 and 3]{Oz1977}, extending work of Ochiai and Takahashi
\cite{OT1976}, so we need only look at the assertion that either
$\Gamma \subset \ell(G)\ell(A)$ or $\Gamma \subset r(G,ds^2)r(A)$.  That is
Proposition \ref{lr4} below.  The theory does have some open problems
concerning homogeneity when $\Gamma \subset \ell(G)$.

\medskip
\addtocounter{subsection}{1}
\centerline{\bf \thesubsection. The Isometry Group.}
\smallskip

We start with a connected Lie group $G$.  $\ell(G)$ will 
denote the group of left translations $\ell(g): x \mapsto gx$, $r(G)$
the right translations $r(g): x \mapsto xg^{-1}$, and $ds^2$ an 
$\ell(G)$--invariant Riemannian metric on $G$.  Obviously $\ell(G)$
is contained in the identity component $\bI^0(G,ds^2)$.

\begin{proposition}\label{och-tak}
{\rm (\cite[Theorems 1 and 2]{OT1976})}
If $G$ is compact then 
$\bI^0(G,ds^2) \subset \ell(G)r(G)$,
and $\ell(G)$ is a normal subgroup of $\bI^0(G,ds^2)$.
\end{proposition}

Write $\<\,,\,\>$ for the inner product on $\gg$ defined by $ds^2$
and $r(G,ds^2) = \{r(g) \in r(G) \mid \Ad(g) \text{ preserves } \<\,,\,\>\}$.
Obviously $r(G)\cap \bI(G,ds^2) \subset r(G,ds^2)$.  Conversely let
$r(v) \in r(G,ds^2)$.  As $\Ad(v)$ preserves 
$\<\,,\,\>$ at $1 \in G$ it preserves $ds^2$ at every $g \in G$,
so $v \in \bI(G,ds^2)$.  Thus
\begin{corollary}\label{och-tak-cor}
If $G$ is compact then $\bI^0(G,ds^2) = \ell(G)r(G,ds^2)$.
\end{corollary}

An automorphism of a compact connected Lie group either preserves each 
simple factor or permutes the simple factors.  Since $\ell(G)$ is normal
in $\bI^0(G,ds^2)$, now

\begin{corollary}\label{no-switch} 
Suppose that $G$ is compact and simple, and that $a \in \bI(G,ds^2)$.
If $\Ad(a)(\ell(G)) \ne \ell(G)$ then $\Ad(a)(\ell(G)) =  r(G)$. In that case 
$(G,ds^2)$ is a compact simple Lie group with bi--invariant Riemannian metric,
so it is an irreducible Riemannian symmetric space.
\end{corollary}
Let $dt^2$ denote a bi-invariant Riemannian metric on
$G$.  Let $L(G,ds^2)$ and $L(G,dt^2)$ denote the respective isotropy
subgroups of $\bI(G,ds^2)$ and $\bI(G,dt^2)$ at $x_0 := 1 \in G$.
They act on the tangent space $\gg$ at $x_0$ by $dh(\xi) = \Ad(h)\xi$.
From \'Elie Cartan's papers \cite{C1927a} and  \cite{C1927b}, 
$L(G,dt^2)\cap \bI^0(G,dt^2) = \{\ell(h)r(h)\mid h \in G\}$, 
and $L(G,dt^2)$ is generated by (i) $L(G,dt^2)\cap \bI^0(G,dt^2)$, 
(ii) the symmetry $s_{x_0}$ and (iii) the automorphisms $\alpha \in \Aut(G)$
that preserve and are outer on the identity component 
$L^0(G,dt^2)$ of $L(G,dt^2)$.  If $(G,ds^2)$ is not a symmetric space, i.e. 
if $ds^2 \ne dt^2$ (up to multiplication by a real constant), now $L(G,ds^2)$ 
is generated by 
(i) $L^0(G,ds^2)$ and (ii) $\{\alpha \in \Aut(G) \mid \alpha \text{ preserves
both $L(G,ds^2)$ and } ds^2\}$.  In particular,
\begin{equation}\label{isotropy-grp}
\bI(G,ds^2) = \ell(G)L(G,ds^2) \subset \ell(G)L(G,dt^2) = \bI(G,dt^2).
\end{equation}

Adapting (\ref{outer}) to our situation, we denote 
\begin{equation}\label{int-out-grp}
\begin{aligned}
&\Int(G,ds^2) = \{\alpha \in \Aut(G) \mid \alpha(L(G,ds^2)) = L(G,ds^2),\, 
	\alpha|_{L(G,ds^2)} \text{ is inner, and }\alpha
        \text{ preserves } ds^2\} \\
&\Out(G,ds^2) = \{\alpha \in \Aut(G) \mid \alpha \text{ preserves both } 
	L(G,ds^2) \text{ and } ds^2\}/\Int(G,ds^2)
\end{aligned}
\end{equation}
Express 
\begin{equation}\label{reps1}
\Out(G,ds^2) = \Int(G,ds^2)\alpha_1 \cup \dots \cup \Int(G,ds^2)\alpha_k
\end{equation}
where the $\alpha_i \in \Aut(G)$ form a set of representatives modulo
inner automorphisms, and further express 
\begin{equation}\label{reps2}
\alpha_i = \Ad(a_i) = \ell(a_i)r(a_i) \text{ for } 1 \leqq i \leqq k.
\end{equation}  
Then we have the full isometry group as follows.

\begin{theorem}\label{iso-grp}
Let $G$ be a compact connected simple Lie group and $ds^2$ a left-invariant
Riemannian metric on $G$.  Suppose that $(G,ds^2)$ is not a 
symmetric space.  Then $\bI(G,ds^2) = \bigcup_{1\leqq i \leqq k}
\{\ell(G)\ell(a_i) \times r(G,ds^2)r(a_i)\}$ where $a_1 = 1$ and 
$\{a_1, \dots , a_k\}$ is given by 
{\rm (\ref{int-out-grp})}, {\rm (\ref{reps1})} and {\rm (\ref{reps2})}.
\end{theorem}

\medskip
\smallskip
\addtocounter{subsection}{1}
\centerline{\bf \thesubsection.  Constant Displacement Isometries.}
\smallskip

From \cite{W1962}, if $\Gamma_{dt^2}$ is a group of constant displacement
isometries on $(G,dt^2)$ then either $\Gamma_{dt^2} \subset \ell(G)r(Z_G)$ or
$\Gamma_{dt^2} \subset \ell(Z_G)r(G)$.  Now let $\Gamma_{ds^2}$ be a group
of constant displacement isometries on $(G,ds^2)$.  Then certainly 
$\Gamma_{ds^2}$ is a group of fixed point free isometries on $(G,dt^2)$ but
the issue is whether every $\gamma \in \Gamma_{ds^2}$ is of constant 
displacement on $(G,ds^2)$.

We will write $\bar r(G,ds^2)$ for $\{w \in G \mid r(w) \in r(G,ds^2)\}$
and $\bar r(w)$ when $r(w) \in r(G,ds^2)\}$.

\begin{lemma}\label{lr1}
Let $\gamma = \ell(ua)r(va)$ be an isometry of constant displacement on 
$(G,ds^2)$ where
$u \in G$, $v \in \bar r(G,ds^2)$, and $a \in \{a_1, \dots , a_k\}$.
Then $va$ commutes with every $G$--conjugate of $ua$, and $ua$ commutes with
every $G$--conjugate of $va$.
\end{lemma}
\begin{proof}
Calculate the displacement $\rho(1,\gamma(1)) = \rho(uv^{-1},1)$ so
$\gamma$ has constant displacement $c_\gamma = \rho(1,uv^{-1})$r,
and $c_\gamma = \rho(g,\gamma(g)) = \rho(g,uaga^{-1}v^{-1})
= \rho(gv,uaga^{-1}) = \rho(gva,uag) = \rho(va,\Ad(g^{-1})(ua))$.
But $\Ad(g^{-1})(ua) = u'_ga \in Ga$ and $\gamma'_g = \ell(u'_ga)r(va)$
is $\ell(G)$--conjugate to $\gamma$, so it has the same constant displacement
$c_\gamma$.  Now the distance from $va$ to any $\ell(G)$--conjugate of $ua$
is the same constant $c_\gamma$\,.

Take a minimizing geodesic $\sigma(t) = va\cdot \exp(t\xi)$ from $va$ 
to $guag^{-1}$.  Then $\sigma'(0)$ is orthogonal to the orbit $\Ad(G)(ua)$
because that orbit lies in the sphere (boundary of the solid sphere) 
of radius $c_\gamma$ with center $va$.
As $\sigma$ is orthogonal to some $\Ad(G)$--orbit it is orthogonal to every
$\Ad(G)$--orbit that it meets.  Thus $\sigma'(1)$ is tangent to the centralizer
of $guag^{-1}$, which is totally geodesic so it contains
$\sigma(0) = va$.  Now $va$ commutes with every $G$--conjugate of $ua$.  But 
$va\cdot guag^{-1} = guag^{-1}\cdot va$ for $g \in G$ so 
$g^{-1}(va\cdot guag^{-1}) g = g^{-1}(guag^{-1}\cdot va)g$ and
$(g^{-1}vag)u = ua(g^{-1}vag)$. Thus $ua$ commutes with 
every $G$--conjugate of $va$.
\end{proof}

\begin{lemma}\label{lr2}
Suppose that $B$ is a simple Lie group and that $B^0a$ is one of its
topological components.  Let $ua, va \in B^0a$ such that $va$ commutes 
with every $B^0$--conjugate of $ua$.  Then either $ua$ or $va$ belongs
to the centralizer $Z_B(B^0)$.
\end{lemma}
\begin{proof}

Suppose $va \notin Z_B(B^0)$.  Then the centralizer $Z_B(va)$ is a proper 
subgroup of lower dimension in $B$ that contains $\Ad(B^0)(ua)$.  But 
$\Ad(B^0)(ua)$
generates a closed normal subgroup $E$ of $B$ contained in $Z_B(va)$.
If $ua \notin Z_B(B^0)$ then $\dim E < \dim B$, contradicting
simplicity of $B$ (as a Lie group).
\end{proof}

Combining Lemmas \ref{lr1} and \ref{lr2} we have
\begin{lemma}\label{lr3}
Let $u \in G$, $v \in \bar r(G,ds^2)$ and $a \in \{a_1, \dots , a_k\}$
such that $\gamma := \ell(ua)r(va)$ is an 
isometry of constant displacement on $(G,ds^2)$.  Write
$\alpha = \Ad(a) = \ell(a)r(a)$, so $\gamma = \ell(u)r(v)\alpha$. 
Then there are two cases:

{\rm (1)} $ua \in Z_{\bI(G,ds^2)}(G)$.  Then $\ell(G)$ centralizes $\gamma$.

{\rm (2)} $va \in Z_{\bI(G,ds^2)}(G)$.  Then $d[\ell(va)r(va)]$ is the
identity on the tangent space $\gg$ so its action on $G$ is trivial and
$\gamma = \ell(uv^{-1})$.
\end{lemma}

Now we see that these cases cannot be mixed.
\begin{proposition}\label{lr4}
Let $\Gamma$ be a finite 
group of isometries of constant displacement on $(G,ds^2)$.  
There are two cases:
\begin{itemize}
\item[(1.)] $\Gamma \subset \bigcup_{1\leqq i \leqq k} \ell(Ga_i)$ 
where $\{\alpha_1, \dots , \alpha_k\}$ and $\{a_1, \dots , a_k\}$ are given by
{\rm (\ref{int-out-grp})}, {\rm (\ref{reps1})} and {\rm (\ref{reps2})}, and
\item[(2.)] $\Gamma \subset \bigcup_{1\leqq i \leqq k} r((G,ds^2)a_i)$.
\end{itemize}
In {\rm Case 2} 
the centralizer of $\Gamma$ in $\bI(G,ds^2)$ contains $\ell(G)$, so the 
Homogeneity Conjecture holds for $\Gamma$.
\end{proposition}

\begin{proof} 
If $\Gamma \not \subset \bigcup_{1\leqq i \leqq k} r((G,ds^2)a_i)$ 
then $\Gamma$ contains an element
$\gamma_1 = \ell(u_1'a_i)r(v_1'a_i)$ with 
$u_1'a_i \notin Z_{\bI(G,ds^2)}(G)$.  Then
$v_1'a_i \in Z_{\bI(G,ds^2)}(G)$ and $\gamma_1 = \ell(u_1)$ where 
$u_1 = (u_1'a_1)(v_1'a_i)^{-1} = u_1'v_1'^{-1}\notin Z_G$\,.  
Then, if we have $\gamma_2 = \ell(u_2'a_j)r(v_2'a_j) \in \Gamma$
with $\gamma_2'a_j \notin Z_G$\,, $\gamma_2 = r(v_2'u_2'^{-1}) = r(v_2)$.
Thus $\gamma_3 := \gamma_1\gamma_2 = \ell(u_1)r(v_2)$ where $v_2 \in Z_G$
because $u_1 \notin Z_G$\,.  Thus $\gamma_3 = \ell(u_1v_2^{-1}) \in \ell(G)$
and $\gamma_1 = \ell(u_1) \in \ell(G)$ so also 
$\gamma_2 = \gamma_1^{-1}\gamma_3 \in \ell(G)$.  
But $\gamma_2 = r(v_2) \notin \ell(G)$.  This contradicts
our hypothesis $\Gamma \not \subset \bigcup_{1\leqq i \leqq k} r((G,ds^2)a_i)$.
We conclude that $\Gamma \subset \bigcup_{1\leqq i \leqq k}\ell(Ga_i)$.
\end{proof}

\section{\bf Positive Curvature Manifolds.}\label{sec10}
\setcounter{equation}{0}
\setcounter{subsection}{0}

In this section we verify the Homogeneity Conjecture for Riemannian
manifolds $(M,ds^2)$ such that there is some Riemannian metric $dt^2$ of
strictly positive sectional curvature on $M$.  Specifically, we combine
Propositions \ref{7-normal}, \ref{conj4s3}, \ref{conj4su} and
\ref{spm-comp} with the comments at the beginning of Subsection \ref{drop}C,

\begin{theorem}\label{conj-non-normal}
Let $M = G/H$ be a connected, simply connected homogeneous space,
and $ds^2$ a $G$--invariant Riemannian metric on $M$, where $ds^2$ is
not required to be the normal Riemannian metric.
Suppose that $M$ admits another invariant Riemannian metric $dt^2$
of strictly positive curvature.  Then the Homogeneity Conjecture
is valid for $(M,ds^2)$.
\end{theorem}

\smallskip
\addtocounter{subsection}{1}
\centerline{\bf \thesubsection. The Classification.}

\noindent
The connected simply connected 
homogeneous Riemannian manifolds of positive sectional curvature were
classified by Marcel Berger \cite{B1961}, Nolan Wallach \cite{W1972},
Simon Aloff and Nolan Wallach \cite{AW1974}, and
Lionel B\' erard-Bergery \cite{B1976}.  Their isometry groups were worked
out by Krishnan Shankar \cite{S2001}.  The spaces and the isometry
groups are listed in the first two columns of Table \ref{shankar-table}
below.  When there is a fibration that will be relevant to our
verification of the Homogeneity Conjecture, it will also be listed in
the first column.

\addtocounter{equation}{1}
\begin{longtable}{|r|l|l|}
\caption*{\bf \centerline{{\normalsize Table} \thetable} 
        \centerline{{\normalsize
        Isometry Groups of CSC Homogeneous Manifolds of}} 
        \centerline{{\normalsize  Positive Curvature and Fibrations 
                over Symmetric Spaces}}} 
\label{shankar-table} \\
\hline
 & $M = G/H$ & $\bI(M,ds^2)$ \\ \hline
\hline
\endfirsthead
\multicolumn{3}{l}{{\normalsize \textit{Table \thetable\, continued from
        previous page $ \dots$}}} \\
\hline
 & $M = G/H$ & $\bI(M,ds^2)$ \\ \hline
\hline
\endhead
\hline \multicolumn{3}{r}{{\normalsize \textit{$\dots$ Table \thetable\,
        continued on next page}}} \\
\endfoot
\hline
\endlastfoot
\hline
{\rm 1} & $S^n = SO(n+1)/SO(n)$ & $O(n+1)$ 
        \\ \hline 
{\rm 2}  & $P^m(\C) = SU(m+1)/U(m)$ & $PSU(m+1)\rtimes \Z_2$ 
        \\ \hline
{\rm 3} & $P^k(\H) = Sp(k+1)/(Sp(k)\times Sp(1))$ & $Sp(k+1)/\Z_2)$ 
        \\ \hline
{\rm 4} & $P^2(\O) = F_4/Spin(9)$ & $F_4$ 
        \\ \hline
{\rm 5} & $S^6 = G_2/SU(3)$ & $O(7)$ 
        \\ \hline
{\rm 6} & $\begin{array}{l} P^{2m+1}(\C) = Sp(m+1)/(Sp(m)\times U(1)) \\ \hline
           P^{2m+1}(\C) \to P^m(\H) \end{array}$ 
        & $(Sp(m+1)/\Z_2)\times \Z_2$  
         \\ \hline
{\rm 7} & $\begin{array}{l} F^6 = SU(3)/T^2 \\ \hline
           F^6 \to P^2(\C) \end{array}$ 
        & $(PSU(3) \rtimes \Z_2) \times \Z_2$  
         \\ \hline
{\rm 8} & $\begin{array}{l} F^{12} = Sp(3)/(Sp(1)\times Sp(1)\times Sp(1))
                \\ \hline F^{12} \to P^2(\H) \end{array}$ 
        & $(Sp(3)/\Z_2) \times \Z_2$  
        \\ \hline
{\rm 9} & $\begin{array}{l} F^{24} = F_4/Spin(8) \\ \hline
           F^{24} \to P^2(\O) \end{array}$ 
        & $F_4$  
        \\ \hline
{\rm 10} & $M^7 = SO(5)/SO(3)$ & $SO(5)$  
        \\ \hline
{\rm 11} & $\begin{array}{l} 
        M^{13} = SU(5)/(Sp(2) \times_{\Z_2} U(1) \\ \hline 
                M^{13} \to P^4(\C) \end{array}$
        & $PSU(5) \rtimes \Z_2$ 
        \\ \hline
{\rm 12} & $N_{1,1} = (SU(3) \times SO(3))/U^*(2)$ & 
                $(PSU(3) \rtimes \Z_2) \times SO(3)$  
        \\ \hline
{\rm 13} & $\begin{array}{l} N_{k,\ell} = SU(3)/U(1)_{k,\ell} \\
           (k,\ell) \ne (1,1), 3|(k^2+\ell^2+k\ell) \\ \hline
           N_{k,\ell} \to P^2(\C) \end{array}$ &
                $(PSU(3)\rtimes \Z_2) \times (U(1)\rtimes \Z_2)$ 
        \\ \hline
{\rm 14} & $\begin{array}{l} N_{k,\ell} = SU(3)/U(1)_{k,\ell} \\
           (k,\ell) \ne (1,1), 3\not| (k^2+\ell^2+k\ell) \\ \hline
           N_{k,\ell} \to P^2(\C) \end{array}$ &
        $U(3)\rtimes \Z_2$ 
        \\ \hline
{\rm 15} & $\begin{array}{l} S^{2m+1} = SU(m+1)/SU(m) \\ \hline
           S^{2m+1} \to P^m(\C) \end{array}$ & 
         $U(m+1)\rtimes\Z_2$ 
        \\ \hline
{\rm 16} & $\begin{array}{l} S^{4m+3} = Sp(m+1)/Sp(m) \\ \hline
           S^{4m+3} \to P^m(\H) \end{array}$ & 
        $Sp(m+1)\rtimes_{\Z_2} Sp(1)$ 
        \\ \hline
{\rm 17} & $\begin{array}{l} S^3 = SU(2)\\ \hline
           S^3 \to P^1(\C) = S^2 \end{array}$ 
        & $O(4)$  
        \\ \hline
{\rm 18} & $S^7 = Spin(7)/G_2$ & $O(8)$  
        \\ \hline
{\rm 19} & $\begin{array}{l} S^{15} = Spin(9)/Spin(7) \\ \hline
           S^{15} \to S^8 \end{array}$ & 
        $Spin(9)$  
        \\ \hline
\end{longtable}

\smallskip
\addtocounter{subsection}{1}
\centerline{\bf \thesubsection. The Normal Metric Case.}

\noindent

The first four spaces $M = G/H$ of Table
\ref{shankar-table} are Riemannian symmetric spaces with $G = \bI(M)^0$.
The fifth space is $S^6 = G_2/SU(3)$, where the isotropy group $SU(3)$ is
irreducible on the tangent space, so the only invariant metric is the one
of constant positive curvature; thus it is isometric
to a Riemannian symmetric space.  In view of \cite{W1962},
the Homogeneity Conjecture is valid for the entries {\rm (1)} through
{\rm (5)} of {\rm Table \ref{shankar-table}}.

Entries (6), (7), (8) and (9) of Table \ref{shankar-table} have
$\rank G = \rank H$.  Each is isotropy--split with
fibration over a projective (thus Riemannian symmetric) space, as defined in
\cite[(1.1)]{W2018}.  The Homogeneity Conjecture follows, for
these $(M,ds^2)$ where $ds^2$ is the normal Riemannian metric.

The argument for entries (6), (7), (8) and (9) applies with only obvious
changes to a number of other table entries, using \cite[Theorem 6.1]{W2018}
instead of \cite[Corollary 5.7]{W2018}.  
And the Homogeneity Conjecture is immediate for (18), where $H = G_2$ acts
irreducibly on the tangent space of $S^7 = Spin(7)/G_2$, so $ds^2$ is the
standard constant positive curvature metric.  
That verifies
the Homogeneity Conjecture for the entries
{\rm (11), (13), (14), (15), (16), (17), (18) and (19)} of
{\rm Table \ref{shankar-table}}, where $ds^2$ is the normal
Riemannian metric on $M$.  We have to take a closer look at 
table entries (10) and  (12)

In case (10), where $M^7 = G/H = SO(5)/SO(3)$,  $H$ acts
irreducibly on the tangent space, so $ds^2$ is normal and naturally 
reductive.  Let $\gg = \gh + \gm$, $\gh \perp \gm$ as usual, and $\gamma$
an isometry of constant displacement.  Since $\bI(M,ds^2) = SO(5)$ is
connected we have $\eta \in \gm$ such that
$\sigma(t) = \exp(t\xi)x_0\,, 0 \leqq t \leqq 1$, is a minimizing geodesic 
in $(M,ds^2)$ from $x_0 = 1H$ to $\gamma(x_0)$. One argues that the
corresponding Killing vector field $X$ on $(M,ds^2)$ has constant length.
But there is no such nonzero vector field \cite{XW2016}.  Thus
there is no isometry $\ne 1$ of constant displacement for entry {\rm (10)}
of {\rm Table \ref{shankar-table}}, so the Homogeneity Conjecture 
is immediate in that case.

In case (12), where $N_{1,1} = G/H = (SU(3)\times SO(3))/U^*(2)$,
let $\gamma$ be an isometry of constant displacement $d > 0$ and
suppose that $\gamma^2$ also is an isometry of constant
displacement.  The argument for table entry (18) shows that 
$\gamma \not\in \bI^0(N_{1,1},ds^2)$, so 
$\gamma = (g_1,g_2)\nu$ where $g_1 \in SU(3)$, $g_2 \in SO(3)$,
$\nu^2 = 1$, $\Ad(\nu)$ is complex conjugation on $SU(3)$, and
$\Ad(\nu)$ is the identity on $SO(3)$.  It also shows 
$\gamma^2 \in \bI^0(N_{1,1},ds^2)$ so $\gamma^2 = 1$.

The centralizer of $\nu$
is $K := ((SO(3) \times SO(3)) \cup (SO(3) \times SO(3))\nu$.
Using
de Siebenthal \cite{dS1956} we reduce our considerations to the 
cases where $g_1$ is either
the identity matrix $I_3$ or the matrix $I'_3 := 
\left ( \begin{smallmatrix} -1 & 0 & 0 \\ 0 & -1 & 0 \\ 0 & 0 & +1 
\end{smallmatrix} \right )$, and also $g_2$ is either $I_3$ or $I'_3$.

Recall that $H = U^*(2)$ is the image of
$U(2) \hookrightarrow (SU(3) \times SO(3))$, given by
$h \mapsto (\alpha(h),\beta(h))$ where $\alpha(h) = 
\left ( \begin{smallmatrix} h & 0 \\ 0 & 1/\det(h) \end{smallmatrix} \right )$
and $\beta$ is the projection $U(2) \to U(2)/(center) \cong SO(3)$.
Further, $\bI(L_{1,1}) = G \cup G\nu$
and its isotropy subgroup is $H \cup H\nu$.  Observe that

\qquad if $(g_1,g_2) = (I_3\,, I_3)$ then $\gamma = 
        (\alpha(I_2),\beta(I_2))\nu \in H\nu$, and

\qquad if $(g_1,g_2) = (I'_3\,, I_3)$ then $\gamma = 
        (\alpha(-I_2),\beta(-I_2))\nu \in H\nu$.

\noindent Replace $I'_3$ by $I''_3 =
\left ( \begin{smallmatrix} -1 & 0 & 0 \\ 0 & +1 & 0 \\ 0 & 0 & -1 
\end{smallmatrix} \right )$ and set $I''_2 =
\left ( \begin{smallmatrix} -1 & 0 \\ 0 & +1 \end{smallmatrix} \right )$, so

\qquad if $(g_1,g_2) = (I''_3\,, I''_3)$ then $\gamma = 
        (\alpha(I''_2),\beta(I''_2))\nu \in H\nu$.

\noindent
When $\gamma \in H\nu$ it cannot be of nonzero constant displacement.
We have reduced our considerations to the case $\gamma = (I_3\,, I'_3)\nu$,
or equivalently to one of its conjugates.  Compute
$$
\left ( \left ( \begin{smallmatrix} i & 0 & 0 \\ 0 & i & 0 \\ 0 & 0 & -1  
\end{smallmatrix} \right ), I_3\right ) \cdot (I_3\,, I'_3)\nu \cdot
\left ( \left ( \begin{smallmatrix} i & 0 & 0 \\ 0 & i & 0 \\ 0 & 0 & -1  
\end{smallmatrix} \right ), I_3 \right )^{-1} \\
 = (I'_3\,, I'_3)\nu  \in H\nu,
$$
so again $\gamma$ cannot be of nonzero constant displacement.
Thus
there is no isometry $\ne 1$ of constant displacement for entry {\rm (12)}
of {\rm Table \ref{shankar-table}}, so the Homogeneity Conjecture
is immediate in that case.

Summarizing this section, we have proved

\begin{proposition}\label{7-normal}
Let $M = G/H$ where $G$ is a compact connected Lie group and $M$ is simply
connected.  Let $ds^2$ be the normal Riemannian metric on $M$.
Suppose that $M$ has a Riemannian metric $dt^2$ for which every 
sectional curvature is $> 0$.  Then the
Homogeneity Conjecture is valid for $(M,ds^2)$.
\end{proposition}

\addtocounter{subsection}{1}
\centerline{\bf \thesubsection. Dropping the Normality Condition.}\label{drop}

\noindent
In this subsection we see how to drop the normality condition
on $ds^2$ in Proposition \ref{7-normal}.  In several cases this is automatic
because the adjoint action of $H$ on the tangent space $\gg/\gh$
is irreducible; there every invariant Riemannian metric on $M = G/H$
is normal. Those are
the spaces given by the entries (1) (2), (3), (4), (5), (10), 
(11), (12) and (18) of Table \ref{shankar-table}.
To consider most of the others one needs a variation on (\ref{split-def}).
\begin{equation}\label{new-setup}
\begin{aligned}
&G \text{ is a compact connected simply connected Lie group, } \\
&H \subset K \text{ are closed connected subgroups of $G$, 
	$M' = G/K$ and $F=H\backslash K$\,, and} \\
& ds^2 \text{ is a $G$--invariant Riemannian metric on $M = G/H$, 
	$ds'^2 = ds^2|_{M'}$ and $ds_F^2 = ds^2|_F$, such that} \\
& \phantom{XX}\text{(i) $\pi: M \to M'$ by  $\pi(gH) = gK$\,,
                right action of $K$\,, } \\
& \phantom{XX}\text{(ii) $(M',ds'^2)$ and $(F,ds_F^2)$ are 
		Riemannian symmetric spaces, and}\\
& \phantom{XX}\text{(iii) the tangent spaces $\gm'$ for $M'$, $\gm''$
        for $F$ and} \text{ $(\gm' + \gm'')$ for $M$ satisfy 
                $\gm' \perp \gm''$}\,.
\end{aligned}
\end{equation}
That leads to a modification of \cite[Lemma 5.2]{W2018}:

\begin{lemma}\label{new-go-space} Assume {\rm (\ref{new-setup})}.
Then the fiber $F$ of
$M \to M'$ is totally geodesic in $M$
In particular it is a geodesic orbit space, and any geodesic of
$M$ tangent to $F$ at some point is of the form
$t \mapsto \exp(t\xi)x$ with $x \in F$ and $\xi \in \gm''$\,.
\end{lemma}
\noindent
That facilitates a variation on the arguments of 
\cite[Proposition 5.4]{W2018}  and \cite[Lemma 5.5]{W2018}
to show that $\Gamma$ centralizes $G$, where $\Gamma$ is a group of
isometries of constant displacement on $(M,ds^2)$, as follows

\begin{lemma}\label{apply-setup}
Suppose that $M = G/H$ is an entry of {\rm Table \ref{shankar-table}} 
for which $G = \bI^0(M, ds^2)$, that $(M, ds^2)$ satisfies
{\rm (\ref{new-setup})}, and that
$\pi: (M,ds^2) \to (M',ds'^2)$ is a Riemannian submersion.  
Then the Homogeneity Conjecture holds for $(M,ds^2)$.
\end{lemma}

With some adjustments, especially for (6) and (19), Lemma 
\ref{apply-setup} applies to (6) (7), (8), (9), (13), (14) and (19).
The remaining three cases are addressed by direct computation.  The simplest,
(17), is the group manifold $SU(2) = Sp(1) = S^3$, described just after
(\ref{mx-condition}).  
There, (1) is the limit of (2) as $a \uparrow 1$ and of (3) as 
$a, b \uparrow 1$.  Writing $(g,h) := \ell(g)r(h)$ for the transformation 
$x \mapsto gxh^{-1}$,
\begin{lemma}
Let $\Gamma \subset \bI(S^3,ds^2)$ be a finite group of 
constant displacement isometries of $(S^3,ds^2)$.  If
$\gamma = \pm (g,h) \in \Gamma$ and $g \ne \pm 1$ then $h = \pm 1$.
\end{lemma}
\begin{corollary}\label{pmG}
Let $\Gamma \subset \bI(S^3,ds^2)$ be a finite group of
constant displacement isometries of $(S^3,ds^2)$.  Then either
$\Gamma \subset [SU(2) \times \{\pm 1\}]/[\pm (1,1)]$ or
$\Gamma \subset [\{\pm 1\} \times H]/[\pm (1,1)]$.
\end{corollary}
\begin{proposition} \label{conj4s3}
Let $ds^2$ be a left $SU(2)$--invariant Riemannian metric on $S^3$.
Let $\Gamma$ be a finite group of isometries of constant displacement
on $(S^3,ds^2)$.  Then the centralizer of $\Gamma$ in $\bI(S^2,ds^2)$
is transitive on $S^3$, so the quotient Riemannian manifold 
$\Gamma \backslash (S^3,ds^2)$ is homogeneous.  In other words, the
Homogeneity Conjecture is valid for $(S^3,ds^2)$.
\end{proposition}

The second remaining case, (15), is the sphere
$G/H = SU(m+1)/SU(m) = S^{2m+1}, m \geqq 2$, total space of a 
circle bundle $S^{2m+1} = G/H \to G/K = P^m(\C)$. The
fiber over $z_0 = 1K$ is the center $Z_K$ of $U(m)$.  $G/H$
has tangent space $\gv \oplus \gz_K$ where $\gv$ is the tangent space
$\C^m$ of $G/K$ and $\gz_K$ is the center of $\gk$;
$\gv$ and $\gz_K$ are the (two) irreducible summands of the
isotropy representation of $H$.  Let $ds^2$ be an $SU(m+1)$--invariant 
Riemannian metric on $M = S^{2m+1}$. Then $\mathbf{I}(M,ds^2)$ is either
the orthogonal group $O(2m+2)$ or $[U(m+1)\cup\nu U(m+1)]$
where $\Ad(\nu)$ is complex conjugation on $U(m+1)$.
In the first case $(M,ds^2)$ is the constant curvature $(2m+1)$--sphere,
where we know that the Homogeneity Conjecture is valid.  In the second
case $ds^2$ is given by
\begin{equation}\label{metric}
ds^2|_{\gv} = b'\kappa|_{\gv}\,,\,\, ds^2|_{\gz_K} = b''\kappa|_{\gz_K}\,,
\,\,\text{ and } \,\,ds^2(\gt',\gz_K) = 0
\end{equation}
for some $b', b'' > 0$.
The displacement satisfies
$c^2 = b'\kappa(\eta',\eta') + b''\kappa(\eta'',\eta'')$.  The normal 
metric is given by $b' = b''$.  After some more computation one arrives at
\begin{lemma}\label{if-nu}
Let $\Gamma \subset \mathbf{I}(M,ds^2)$ be a subgroup such that every
$\gamma \in \Gamma$ is an isometry of constant displacement.  If
$\gamma = \nu g \in \Gamma\cap \nu U(m+1)$ then $m+1$ is even,
$\gamma^2 = -I \in U(m+1)$, and $\Gamma$ is $SU(m+1)$--conjugate to the 
binary dihedral group whose centralizer in $U(m+1)$ is $Sp(\tfrac{m+1}{2})$.
\end{lemma}
\begin{proposition}\label{conj4su}
Let $ds^2$ be an $SU(m+1)$--invariant Riemannian metric on $S^{2m+1}, 
m \geqq 2$.  Let $\Gamma$ be a finite group of isometries of constant
displacement on $S^{2m+1}$.  Then the centralizer of $\Gamma$ in
$\mathbf{I}(S^{2m+1}, ds^2)$ is transitive on $S^{2m+1}$, so the
Riemannian quotient manifold $\Gamma\backslash (SU(m+1), ds^2)$ is
homogeneous.  In other words, the Homogeneity Conjecture is valid for
$(S^{2m+1}, ds^2)$.
\end{proposition}

The third and most delicate remaining case, (16), is the sphere 
$G/H = S^{4m+3}$, total space of an $S^3$ bundle 
$G/H = Sp(m+1)/Sp(m) \to Sp(m+1)/ \{Sp(m) \times Sp(1)\} = G/K$.  
$G/H$ has tangent space $\gv + \gw$ where $\gv$ is the
tangent space $\H^m$ of $G/K$ and $\gw$ is the tangent space $\Im \H$
of the fiber of $S^{4m+3} \to P^m(\H)$.  The isotropy representation
of $H$ is the natural representation of $Sp(m)$ on $\H^m = \gv$, and
on $\gw$ it is three copies of the trivial representation.

Write $\kappa$ for the negative of the Killing form of $\gg$.
Let $\kappa' = \kappa|_\gv$ and $\kappa'' = \kappa|_\gw$ where 
$\kappa(\mu,\nu) = -\Re\tr(\mu\,\overline{\nu})$ with trace taken in $Sp(m+1)$.
Let $\{e_1,e_2,e_3\}$ be a $\kappa''$--orthonormal basis of $\gw$
and split $\kappa'' = \kappa_1 + \kappa_2 + \kappa_3$ accordingly.
Then
\begin{equation}\label{sp-metric}
ds^2|_\gv = b_0\kappa',\, ds^2|_\gw = b_1\kappa_1 + b_2\kappa_2 + 
        b_3\kappa_3\,,\, ds^2(\gv,\gw) = 0 \text{ and }
        ds^2(e_i,e_j) = 0 \text{ for } i\ne j
\end{equation}
for some positive numbers $b_0, b_1, b_2 \text{ and } b_3$\,.

\begin{lemma} \label{sp2o}
Let $ds^2$ be an $Sp(m+1)$--invariant Riemannian metric on $M = S^{4m+3}$,
$m \geqq 1$.
Then either $ds^2$ is invariant under $SU(2m+2)$, or 
$\mathbf{I}(M,ds^2) = Sp(m+1)\cdot L = (Sp(m+1)\times L)/\{\pm (I_{m+1},I_3)\}$
where $L$ is one of the following.

{\rm (1)} $L = Sp(1)$ acting on $Sp(m+1)/Sp(m)$ on the right.  $L$ acts
on the tangent space as multiplication by quaternion unit scalars
on $\gv$ and the adjoint representation of $Sp(1)$ on $\gw$.  This is
the case $b_1 = b_2 = b_3$\,.

{\rm (2)} $L = O(2) \times \Z_2$ acting on $Sp(m+1)/Sp(m)$ on the right.  
$L$ acts the tangent space as multiplication by an $O(2)\times \Z_2$ 
{\rm (}essentially circle{\rm )} group of 
quaternion unit scalars on $\gv$, $O(2)$--rotation on the $(e_1,e_2)$--plane 
in $\gw$, and the $\Z_2$--action $e_3 \mapsto \pm e_3$ on $\gw$\,.  
This is the case where two, but not all three, of the $b_i$ are equal,
for example where $b_1 = b_2 \ne b_3$\,.

{\rm (3)} $L = \Z_2^3$ acting on $Sp(m+1)/Sp(m)$ on the right.  $L$
acts on the tangent space by $\pm 1$ on $\gv$ and the $e_i \mapsto \pm e_i$
on $\gw$.  This is the case where $b_1$\,, $b_2$ and $b_3$ are all different.
\end{lemma}

Making use of (\ref{new-setup}), the particular fibration here, and 
properties of $P^m(\H)$, one arrives at
\begin{lemma}\label{sp-comp}
If $\gamma \in Sp(m+1)$ has constant displacement $c > 0$
on $S^{4m+3}$, $m \geqq 2$, then $\gamma$ belongs to the centralizer
of $Sp(m+1)$ in $\mathbf{I}(S^{4m+3},ds^2)$.
\end{lemma}
\begin{proposition}\label{spm-comp}
Let $\Gamma \subset {\mathbf I}(S^{4m+3},ds^2)$ be a subgroup such that every
$\gamma \in \Gamma$ is an isometry of constant displacement.  Suppose
$m \geqq 2$ and that $ds^2$ is not $SU(2m+2)$--invariant.  Then $\Gamma$
centralizes $Sp(m+1)$ in ${\mathbf I}(M,ds^2)$.
\end{proposition}

As noted at the beginning of Section \ref{sec10}, Propositions 
\ref{7-normal}, \ref{conj4s3}, \ref{conj4su} and \ref{spm-comp}, together
with the comments at the beginning of Subsection \ref{drop}C, combine to
give the main result of this section, Theorem \ref{conj-non-normal}.

\medskip
\centerline{\bf Part III.  Noncompact Homogeneous Riemannian Manifolds.}
\medskip

In the next four sections we will sketch the proof of the Homogeneity 
Conjecture for several classes of noncompact Riemannian homogeneous spaces.
In these noncompact cases ``bounded'' can replace ``constant displacement''
and the result becomes independent of the choice of Riemannian metric.

\section{\bf Negative Curvature.}\label{sec11}
\setcounter{equation}{0}
\setcounter{subsection}{0}

If $\gamma$ is a bounded isometry of a connected, simply connected, 
Riemannian manifold $(M,ds^2)$ of sectional curvature $\leqq 0$, then 
\cite{W1964} $\gamma$ is an ordinary translation along the euclidean
factor in the de Rham decomposition of $(M,ds^2)$.  Thus, as we 
noted in Proposition \ref{ill-j},

\begin{quote}{\rm (\cite[Theorem 1]{W1964})} If $(L,dt^2)$
is a complete connected simply connected Riemannian manifold of sectional
curvature $\leqq 0$, with no euclidean factor in its de Rham decomposition,
then every bounded isometry of $(L,dt^2)$ is trivial.  In particular if a
Riemannian quotient $(L',dt'^2) := \Gamma \backslash (L,dt^2)$ is homogeneous
then $\Gamma = \{1\}$ and $(L',dt'^2) = (L,dt^2)$.
\end{quote}

The Homogeneity Conjecture follows immediately for Riemannian manifolds
of sectional curvature $\leqq 0$.  There is an extension, due to Druetta,
to manifolds without focal points \cite{D1983}, and the Homogeneity
Conjecture is immediate for those spaces as well.

\medskip
\section{\bf Semisimple Groups.}\label{sec12}
\setcounter{equation}{0}
\setcounter{subsection}{0}

Suppose that $G'$ is a connected real semisimple Lie group without compact 
local factors. Consider a Riemannian manifold $(M,ds^2)$ on which $G'$ acts
transitively and effectively by isometries, in other words 
$G' \subset \bI(M,ds^2)$ is transitive on $M$.  An isometry $\gamma$ of
$(M,ds^2)$ is {\sl bounded} if the displacement function $c_\gamma(x) := 
\rho(x,\gamma(x))$ is bounded.

\begin{proposition}\label{DMW1} {\rm \cite[Theorem 2.1]{MMW1986}.}
The centralizer $B:= Z_{\bI(M,ds^2)}(G')$ of $G'$ in $\bI(M,ds^2)$ is the 
set of all bounded isometries of $(M,ds^2)$.  In particular every bounded 
isometry of $(M,ds^2)$ centralizes $G'$ and thus is of constant displacement
on $(M,ds^2)$.  Thus the Homogeneity Conjecture holds for $(M,ds^2)$.
\end{proposition}

The structure of $B$ in Proposition \ref{DMW1} is given as follows
\cite[Section 2]{MMW1986}.  Let $G$ denote the closure of $G'$
in $\bI(M,ds^2)$.  Then $G'$ is the derived group of $G$, and $G$ 
is reductive.  Express $M = G/H$ where $H$ is the isotropy subgroup
at some point $x_0 \in M$.  $H$ is compact because $G$ is transitive 
on $M$ and is closed in $\bI(M,ds^2)$. Write $N_G(H)$ for the normalizer
of $H$ in $G$ and consider the right translations
$$
r(u): gH \mapsto gu^{-1}H \text{ for } u \in N_G(H) \text{ and }
U = \{u \in N_G(H) \mid r(u) \in \bI(M,ds^2)\}.
$$
Of course $r(U) = \{r(u) \mid u \in U\}$.
\begin{proposition}\label{DMW2}
If $\gamma \in \bI(M,ds^2)$ then the following conditions are equivalent.
\begin{itemize}
\item[(1.)] $\gamma$ is an isometry of constant displacement on $(M,ds^2)$.
\item[(2.)] $\gamma$ is an isometry of bounded displacement on $(M,ds^2)$.
\item[(3.)] $\gamma \in r(U)$.
\item[(4.)] The centralizer $Z_{\bI(M,ds^2)}(\gamma)$ is transitive on $M$.
\end{itemize}
In particular, if $\Gamma$ is a discrete subgroup of $\bI(M,ds^2)$
consisting of isometries of constant displacement then 
$\Gamma \backslash (M,ds^2)$ is homogeneous; so the Homogeneity 
Conjecture holds for $(M,ds^2)$.
\end{proposition}

Proposition \ref{DMW1} follows from Proposition \ref{DMW2}.  The proof of
Proposition \ref{DMW2} makes use of \cite[Theorem 4.4]{G1980} and a 
variation on some results of Tits \cite{T1964} which we will describe 
in Section \ref{sec13}.

Let $G_\C$ denote a complex reductive Lie group, $G$ a real form of $G_\C$
and $Q$ a parabolic subgroup of $G_\C$.  Then $Z = G_\C/Q$ is a {\em complex
flag manifold}.  The number of $G$--orbits on $Z$ is finite, so there are 
open orbits.  The open orbits are {\em flag domains} and their structure is
$G/L$ where $L_\C$ is the reductive part of $Q$.  See \cite{W1969} for 
details, \cite{W1974} or \cite{W2018} for applications to the representation 
theory of real semisimple Lie groups, and \cite{WeW1977}, \cite{W1979} and
\cite{WaW1983} for applications to automorphic cohomology theory.
From either Proposition \ref{DMW1} or \ref{DMW2},
\begin{corollary}\label{flag}
Let $D = G/L$ be a flag domain with $L$ compact, and let $ds^2$ be any
$G$--invariant Riemannian metric on $D$.  Then the Homogeneity Conjecture
holds for $(D,ds^2)$.
\end{corollary}

\medskip
\section{\bf Bounded Automorphisms.}\label{sec13}
\setcounter{equation}{0}
\setcounter{subsection}{0}

An automorphism $\alpha$ of a locally compact group $G$ is {\sl bounded}
if there is a compact subset $C \subset G$ such that 
$\alpha(g)g^{-1} \in C$ for every $g \in G$.  If $g \in G$ the inner 
automorphism $\Ad(g):t \mapsto gtg^{-1}$ is bounded if and only if the
conjugacy class $\Ad(G)g$ is relatively compact.  We write $B(G)$ for
the set of all such elements of $G$.  It is a subgroup.
The result of Jacques Tits mentioned in Section \ref{sec12} is

\begin{proposition}\label{Jtits1} {\rm \cite[Th\' eor\` eme (1)]{T1964}}
Soit $G$ un groupe de Lie sans sous-groupe invariant compact non discret,
et soit $N$ son plus grand sous-groupe invariant nilpotent connexe. Alors, 
$B(G)$ est contenu dans le centralisateur $Z_G(M)$ du groupe $M$ engendr\' e 
par $N$ et par tous les sous-groupes simples non compacts de $G$. Les 
composantes connexes de l’\' el\'ement neutre dans $B(G)$ et $Z_G(M)$,
soient $B^0(G)$ et $Z_G^0(M)$, sont des sous-espaces vectoriels du centre 
de $N$ (en particulier, $Z_G^0(M)$ est le centre connexe de $Z_G(M))$; 
de plus, $B^0(G) = B(G)\cap N$. Enfin, si $G$ est connexe et si
$C(G)$ designe son centre, $B(G) = B^0(G)\cdot C(G)$.
\end{proposition}

We extract the part that is relevant for us here:

\begin{corollary}\label{Jtits2}{\rm \cite[Corollaire (2)]{T1964}}
Soit $G$ un groupe de Lie connexe. Si le radical $R$ de $G$ est nilpotent 
et simplement connexe, et si $G/R$ n’a pas de sous-groupe invariant compact 
non discret, $G$ n’a pas d’ automorphisme non trivial \` a d\' eplacement 
born\' e.
\end{corollary}

\medskip
\section{\bf Exponential Solvable Groups.}\label{sec14}
\setcounter{equation}{0}
\setcounter{subsection}{0}

Complementing Corollary \ref{Jtits2},

\begin{proposition}\label{esolv1} 
{\rm (\cite[Theorem 2.5]{W2017} and \cite{W2022})}. 
Let $(M, d)$ be a metric space on which an exponential solvable
Lie group $S$ acts effectively and transitively by isometries. Let 
$G = \bI(M, d)$. Then $G$ is a Lie group, any isotropy subgroup $K$ is 
compact, and $G = SK$. If $g \in G$ is a bounded isometry then $g$ is a 
central element in $S$.
\end{proposition}

We combine Corollary \ref{Jtits2} and Proposition \ref{esolv1}:
\begin{theorem}\label{esolv2}
Let $\alpha$ be a bounded automorphism of a connected Lie group $G$.
Suppose that the solvable radical of $G$ is exponential solvable.
Then $G/\Ker(\alpha)$ is compact.
If the semisimple group $G/R$ has no compact simple factor then
$\alpha$ is the identity.
\end{theorem}
\begin{proof}
Consider a Levi-Whitehead decomposition $G = RS$ where $R$ is the solvable
radical and $S$ is a semisimple complement.  Then $\alpha(R) = R$ and
$\alpha|_R$ is bounded, so $\alpha|_R = 1$ by Proposition \ref{esolv1}.
Now $\alpha$ passes down to $\bar\alpha \in \Aut(G/R)$, where it is a 
bounded automorphism.

Split $S = S_nS_c$ where $S_n$ is the product of the noncompact simple 
normal subgroups and $S_c$ is the product of the compact ones.  Then
$G/R = \bar S_n\bar S_c$ where $\bar S_n$ is the image of $S_n$ under
$G \to G/R$ and $\bar S_c$ is the image of $S_c$\,.  The two factors
of $G/R$ are $\bar\alpha$--invariant, and $\bar\alpha|_{\bar S_n} = 1$
by Proposition \ref{esolv1}.

Now consider a basis $\{u_i,v_j,w_k\}$ of $\gg$ where $\{u_i\}$ is a 
basis of $\gs_n$\,, $\{v_j\}$ is a basis of $\gs_c$\,, and $\{w_k\}$ is a
basis of $\gr$.  The corresponding block form of the matrix of $d\alpha$
on $\gg$ is 
$\left ( \begin{smallmatrix}    0    &    0    & 0 \\
                                0    & A_{2,2} & 0 \\
                             A_{3,1} & A_{3,2} & 0    
         \end{smallmatrix} \right )$
and of $d\alpha|_\gj$ on $\gj := \gs_n + \gr$ is
$\left ( \begin{smallmatrix}    0    & 0 \\
                             A_{3,1} & 0     \end{smallmatrix} \right )$.
That is nilpotent.  Thus a bounded automorphism of $S_nR$ is unipotent,
and consequently is the identity.  Now $A_{3,1} = 0$, so 
$(\gs_n + \gr) \subset \Ker(d\alpha)$.  In other words $\gg/\Ker(d\alpha)$
is a quotient of $\gs_c$ and $G/\Ker(\alpha)$ is a quotient of the compact
group $S_c$\,.  The theorem follows.
\end{proof}

If $g \in \bI(M,ds^2)$ and $(M,ds^2)$ is homogeneous, then 
$g$ is a bounded isometry of $(M,ds^2)$ if and only if $\Ad(g)$
is a bounded automorphism of $\bI^0(M,ds^2)$.  Thus 
an immediate consequence of Theorem \ref{esolv2} is

\begin{theorem}\label{esolv3}
Let $(M,ds^2)$ be a Riemannian manifold on which a connected Lie group $G$ 
acts effectively and transitively by isometries.  
Suppose that the solvable radical $R$ of $G$ is exponential solvable and that
the semisimple quotient $G/R$ has no compact simple factor.  Let $\Gamma$
the a discrete group of isometries of bounded displacement on $(M,ds^2)$.
Then $G$ centralizes $\Gamma$ in $\bI(M,ds^2)$, so $\Gamma \backslash (M,ds^2)$
is homogeneous, and the Homogeneity Conjecture holds for $(M,ds^2)$.
\end{theorem}

\medskip
\centerline{\bf Part IV. Some Problems.}
\medskip

\section{\bf Open Problems.}\label{sec15}
\setcounter{equation}{0}
\setcounter{subsection}{0}
\medskip

We mention a few open problems in connection with the Homogeneity 
Conjecture.  Of course we would welcome solid general proof and we 
hardly need mention that.
In this section we describe five problems that are related to the 
cases where we do have a proof.  The first two are simply to fill gaps in
our knowledge, but the last three would certainly require some new ideas.

\medskip
\addtocounter{subsection}{1}
\centerline{\bf \thesubsection. Outer Automorphisms and 
Non-Normal Metrics For $\chi(M) > 0$.}
\medskip

Consider the case $\chi(M) > 0$ of Section \ref{sec8}.  There $M = G/K$
with $\rank G = \rank K$, and the case of a normal (from the negative of the
Killing form of $\gg$) metric $dt^2$ was settled in Theorem \ref{cw-cover}.
As long as no outer automorphisms of $G$ occur in $\bI(M,ds^2)$ the
Homogeneity Conjecture is proved in the affirmative; see Theorem 
\ref{euler-pos}.  We don't yet know how to deal with outer automorphisms that 
might occur in isometries of constant displacement, but here is one possibility.

There is an $\Ad(K)$--invariant reductive decomposition $\gg = \gk + \gm$,
and since $\gk$ contains a Cartan subalgebra of $\gg$ we can express $\gm$ 
as a sum of root planes 
$\gg^\varphi := (\gg_\C^\varphi + \gg_\C^{-\varphi})\cap \gg$
for an appropriate set of positive roots $\varphi$.  The Borel-de Siebenthal
theory \cite{BdS1949}, applied recursively, should make this explicit and
should give a useful description of $ds^2$ on $\gm$, perhaps using some of the 
decomposition techniques from \cite{GW1968a} and \cite{GW1968b}.

\medskip
\addtocounter{subsection}{1}
\centerline{\bf \thesubsection. Outer Automorphisms for Group Manifolds.}

We are in the case of a compact group manifold $G$ with a left invariant 
Riemannian metric $ds^2$.  If $dt^2$ is the normal metric defined by a
negative multiple of the Killing form of $\gg$, then $(G,dt^2)$ is a
Riemannian symmetric space, and we know the groups $\Gamma$ of constant 
displacement isometries from \cite[Theorem 4.5.1]{W1962}; see Theorem 
\ref{symm-group} in Section \ref{sec4}.  More generally for $(G,ds^2)$, 
in the notation of Proposition \ref{lr4} and Theorem \ref{result-grp}, 
either $\Gamma \subset \bigcup \ell(G)\ell(a_i)$ or 
$\Gamma \subset \bigcup r((G,ds^2)a_i)$.  In this second case $\Gamma$
centralizes the transitive group $\ell(G)$, but in the first case we need
a better understanding of $\ell(\gamma)$ for $\gamma \in \bI(G,ds^2)$
of constant displacement.
A general $G$--invariant metric $ds^2$ can be diagonalized relative to
the normal metric $dt^2$ by diagonalizing the corresponding inner product
$\<\, ,\, \>_s$ on $\gg$,
say $\gg = \sum_{1\leqq i\leqq r} a_i\gm_i$ where the $\gm_i$ are mutually
orthogonal relative to the inner product $\<\, ,\, \>_t$ from $dt^2$
and $0 < a_1 < \dots < a_r$.  Let $x_0$ denote the base point $1$ and consider
a constant displacement isometry $\gamma \in \bI(M,ds^2)$.  Let $\xi \in \gg$
such that $t \mapsto \exp(t\xi)x_0$ is a shortest (for $\xi \in \gg$) curve
from $x_0$ to $\gamma(x_0)$.  Decompose 
$\xi = \sum \xi_i \in \gg$ with $\xi_i \in \gm_i$\,.  Each 
$\exp(t\xi_i)x_0$, $0 \leqq t \leqq 1$, is an $(M,ds^2)$--geodesic, 
and it should be possible to fit them together to describe minimizing 
geodesics in $(M,ds^2)$ from $x_0 = 1K$ to $\gamma(x_0)$. 

\medskip
\addtocounter{subsection}{1}
\centerline{\bf \thesubsection. Weakly Symmetric Spaces and Geodesic Orbit
Spaces.}

Without going into the definitions, weakly symmetric spaces form a very 
interesting class of homogeneous Riemannian manifolds, and geodesic orbit
spaces form a larger interesting class of homogeneous Riemannian manifolds.
It would be important to verify the Homogeneity Conjecture for them.

\medskip
\addtocounter{subsection}{1}
\centerline{\bf \thesubsection. Extension to Finsler Manifolds.}

The notion of ``constant displacement'' makes perfect sense for isometries
of metric spaces.  Furthermore, the proof of the obvious part of the 
Homogeneity Conjecture 
\begin{quote}
Consider a connected homogeneous Riemannian manifold $(M,ds^2)$ and a
Riemannian covering $(M,ds^2) \to \Gamma \backslash (M,ds^2)$.  If
$\Gamma \backslash (M,ds^2)$ is homogeneous then every $\gamma \in \Gamma$
is an isometry of constant displacement.
\end{quote}
holds for metric spaces:
\begin{quote}
Let $x, y \in M$ and $\gamma \in \Gamma$.  Choose $g \in G$ with
$g(x) = y$.  Let $\rho$ denote distance in $(M,ds^2)$.  Then the
displacement $\rho(x,\gamma(x)) =
\rho(gx,g\gamma(x)) = \rho(gx,\gamma g(x)) = \rho(y,\gamma(y))$.
\end{quote}
Thus one can conjecture the converse, as for the Riemannian case.  But that
certainly is asking too much, and one should start by asking whether the
Homogeneity Conjecture is valid for Finsler spaces, or even Finsler spaces
of Berwald type or of Randers type.  Much of the Riemannian geometry 
machinery behind Ozols' result 
\cite[Theorem 1.6]{O1974} (see Proposition \ref{oz1}
above) is available for Finsler spaces and is conveniently set out in
\cite[Chapter 1]{D2012}.
\medskip

\addtocounter{subsection}{1}
\centerline{\bf \thesubsection. Extension to Pseudo-Riemannian Manifolds.}
\medskip

Here the notion of distance -- and thus of displacement -- does not make sense,
but geodesics and the notion of preserving a geodesic are the same as in the
Riemannian case.  The notion of ``minimizing geodesic'' is not available, so
in view of \cite[Theorem 1.6]{O1974} (Proposition \ref{oz1}) one would
look at isometries $\gamma \in \bI(M,ds^2)$ that preserve every geodesic
$\overline{x,\gamma(x)}$, $x \in M$.  The proof the (2) $\Leftrightarrow$ (3)
in \cite[Theorem 1.6]{O1974} would need some adjustment, and it might be 
necessary to restrict attention to geodesic orbit spaces.

\end{document}